\documentclass[11pt]{article}

\usepackage{amssymb,amsmath,euscript,bbm,xcolor,graphicx,epstopdf}
\newtheorem{proposition}{Proposition}[section]
\newtheorem{lemma}[proposition]{Lemma}
\newtheorem{theorem}[proposition]{Theorem}

\newtheorem{corollary}[proposition]{Corollary}

\def\la{\lambda}
\def\La{\Lambda}
\def\ep{\varepsilon}

\def\l{{\langle}}
\def\r{\rangle}

\newcommand{\wt}{\widetilde}
\def\R{{\mathbb R}}

\def\E{{\mathbb E}}
\def\P{{\mathbb P}}

\makeatletter \@addtoreset{equation}{section} \makeatother
\newcommand {\qed}%
{%
    {}\hfill
    {}\hfill
    {$\square $}%
    \vspace {0.3cm}%
    \pagebreak [2]%
    \par
}%
\newenvironment{proof}[1]{%
    \vspace{0.3cm}%
    \pagebreak [2]%
    \par%
    \noindent {\bf  Proof~#1\ }}{\qed}%
\newenvironment{example}{%
    \vspace{0.3cm} \pagebreak [2]%
    \par%
    \refstepcounter{proposition}%
    \noindent%
    {\bf  Example~\theproposition\ }}{}%
\newenvironment{remark}{%
    \vspace{0.3cm} \pagebreak [2]%
    \par%
    \refstepcounter{proposition}
    \noindent%
    {\bf Remark~\theproposition\  }}{}%
\begin{document}

\title {Distribution of the Height of Local Maxima of Gaussian Random Fields \thanks{Research partially
supported by NIH grant R01-CA157528.}}
\author{Dan Cheng\\ North Carolina State University
 \and Armin Schwartzman \\ North Carolina State University }

\maketitle

\begin{abstract}
Let $\{f(t): t\in T\}$ be a smooth Gaussian random field over a parameter space $T$, where $T$ may be a subset of Euclidean space or, more generally, a Riemannian manifold. We provide a general formula for the distribution of the height of a local maximum $\P\{f(t_0)>u | t_0 \text{ is a local maximum of } f(t) \}$ when $f$ is non-stationary. Moreover, we establish asymptotic approximations for the overshoot distribution of a local maximum $\P\{f(t_0)>u+v | t_0 \text{ is a local maximum of } f(t) \text{ and } f(t_0)>v\}$ as $v\to \infty$. Assuming further that $f$ is isotropic, we apply techniques from random matrix theory related to the Gaussian orthogonal ensemble to compute such conditional probabilities explicitly when $T$ is Euclidean or a sphere of arbitrary dimension. Such calculations are motivated by the statistical problem of detecting peaks in the presence of smooth Gaussian noise.
\end{abstract}


\section{Introduction}
In certain statistical applications such as peak detection problems [cf. Schwartzman et al. (2011) and Cheng and Schwartzman (2014)], we are interested in the tail distribution of the height of a local maximum of a Gaussian random field. This is defined as the probability that the height of the local maximum exceeds a fixed threshold at that point, conditioned on the event that the point is a local maximum of the field. Roughly speaking, such conditional probability can be stated as
\begin{equation}\label{Eq:Palm t0}
\P\{f(t_0)>u | t_0 \text{ is a local maximum of } f(t) \},
\end{equation}
where $\{f(t): t\in T\}$ is a smooth Gaussian random field parameterized on an $N$-dimensional set $T\subset\R^N$ whose interior is non-empty, $t_0\in \overset{\circ}{T}$ (the interior of $T$) and $u\in \R$. In peak detection problems, this distribution is useful in assessing the significance of local maxima as candidate peaks. In addition, such distribution has been of interest for describing fluctuations of the cosmic background in astronomy [cf. Bardeen et al. (1985) and Larson and Wandelt (2004)] and describing the height of sea waves in oceanography [cf. Longuet-Higgins (1952, 1980), Lindgren (1982) and Sobey (1992)].

As written, the conditioning event in (\ref{Eq:Palm t0}) has zero probability. To make the conditional probability well-defined mathematically, we follow the original approach of Cramer and Leadbetter (1967) for smooth stationary Gaussian process in one dimension, and adopt instead the definition
\begin{equation}\label{Eq:Palm Ut0}
F_{t_0}(u) :=\lim_{\ep\to 0} \P\{f(t_0)>u | \exists \text{ a local maximum of } f(t) \text{ in } U_{t_0}(\ep) \},
\end{equation}
if the limit on the right hand side exists, where $U_{t_0}(\ep)=t_0\oplus(-\ep/2,\ep/2)^N$ is the $N$-dimensional open cube of side $\ep$ centered at $t_0$. We call (\ref{Eq:Palm Ut0}) the distribution of the height of a local maximum of the random field.

Because this distribution is conditional on a point process, which is the set of local maxima of $f$, it falls under the general category of Palm distributions [cf. Adler et al. (2012) and Schneider and Weil (2008)]. Evaluating this distribution analytically has been known to be a difficult problem for decades. The only known results go back to Cramer and Leadbetter (1967) who provided an explicit expression for one-dimensional stationary Gaussian processes, and Belyaev (1967, 1972) and Lindgren (1972) who gave an implicit expression for stationary Gaussian fields over Euclidean space.

As a first contribution, in this paper we provide general formulae for \eqref{Eq:Palm Ut0} for non-stationary Gaussian fields and $T$ being a subset of Euclidean space or a Riemannian manifold of arbitrary dimension. As opposed to the well-studied global supremum of the field, these formulae only depend on local properties of the field. Thus, in principle, stationarity and ergodicity are not required, nor is knowledge of the global geometry or topology of the set in which the random field is defined. The caveat is that our formulae involve the expected number of local maxima (albeit within a small neighborhood of $t_0$), so actual computation becomes hard for most Gaussian fields except, as described below, for isotropic cases.

We also investigate the overshoot distribution of a local maximum, which can be roughly stated as
\begin{equation}\label{Eq:overshoot t0}
\P\{f(t_0)>u+v | t_0 \text{ is a local maximum of } f(t) \text{ and } f(t_0)>v \},
\end{equation}
where $u>0$ and $v \in \R$. The motivation for this distribution in peak detection is that, since local maxima representing candidate peaks are called significant if they are sufficiently high, it is enough to consider peaks that are already higher than a pre-threshold $v$. As before, since the conditioning event in (\ref{Eq:overshoot t0}) has zero probability, we adopt instead the formal definition
\begin{equation}\label{Eq:overshoot Ut0}
\begin{split}
\bar{F}_{t_0}(u, v) :=\lim_{\ep\to 0} \P\{f(t_0)>u+v | \exists \text{ a local maximum of } f(t) \text{ in } U_{t_0}(\ep) \text{ and } f(t_0)>v\},
\end{split}
\end{equation}
if the limit on the right hand side exists. It turns out that, when the pre-threshold $v$ is high, a simple asymptotic approximation to (\ref{Eq:overshoot Ut0}) can be found because in that case, the expected number of local maxima can be approximated by a simple expression similar to the expected Euler characteristic of the excursion set above level $v$ [cf. Adler and Taylor (2007)].

The appeal of the overshoot distribution had already been realized by Belyaev (1967, 1972), Nosko (1969, 1970a, 1970b) and Adler (1981), who showed that, in stationary case, it is asymptotically equivalent to an exponential distribution. In this paper we give a much tighter approximation to the overshoot distribution which, again, depends only on local properties of the field and thus, in principle, does not require stationarity nor ergodicity. However, stationarity does enable obtaining an explicit closed-form approximation such that the error is super-exponentially small. 
In addition, the limiting distribution has the appealing property that it does not depend on the correlation function of the field, so these parameters need not be estimated in statistical applications.

As a third contribution, we extend the Euclidean results mentioned above for both (\ref{Eq:Palm Ut0}) and (\ref{Eq:overshoot Ut0}) to Gaussian fields over Riemannian manifolds. The extension is not difficult once it is realized that, because all calculations are local, it is essentially enough to change the local geometry of Euclidean space by the local geometry of the manifold and most arguments in the proofs can be easily changed accordingly. 

As a fourth contribution, we obtain exact (non-asymptotic) closed-form expressions for isotropic fields, both on Euclidean space and the $N$-dimensional sphere. This is achieved by means of an interesting recent technique employed in Euclidean space by Fyodorov (2004), Aza\"is and Wschebor (2008) and Auffinger (2011) involving random matrix theory. The method is based on the realization that the (conditional) distribution of the Hessian $\nabla^2 f$ of an isotropic Gaussian field $f$ is closely related to that of a Gaussian Orthogonal Ensemble (GOE) random matrix. Hence, the known distribution of the eigenvalues of a GOE is used to compute explicitly the expected number of local maxima required in our general formulae described above.
As an example, we show the detailed calculation for isotropic Gaussian fields on $\R^2$. Furthermore, by extending the GOE technique to the $N$-dimensional sphere, we are able to provide explicit closed-form expressions on that domain as well, showing the two-dimensional sphere as a specific example.

The paper is organized as follows. In Section \ref{Section:general Euclidean}, we provide general formulae for both the distribution and the overshoot distribution of the height of local maxima for smooth Gaussian fields on Euclidean space. The explicit formulae for isotropic Gaussian fields are then obtained  by techniques from random matrix theory. Based on the Euclidean case, the results are then generalized to Gaussian fields over Riemannian manifolds in Section \ref{section:general manifolds}, where we also study isotropic Gaussian fields on the sphere. Lastly, Section \ref{Section:proofs of main results} contains the proofs of main theorems as well as some auxiliary results.

\section{Smooth Gaussian Random Fields on Euclidean Space}\label{Section:general Euclidean}
\subsection{Height Distribution and Overshoot Distribution of Local Maxima}
Let $\{f(t): t\in T\}$ be a real-valued, $C^2$ Gaussian random field parameterized on an $N$-dimensional set $T\subset\R^N$ whose interior is non-empty. Let
\begin{equation*}
\begin{split}
f_i(t)&=\frac{\partial f(t)}{\partial t_i}, \quad \nabla f(t)= (f_1(t), \ldots, f_N(t))^T, \quad \La(t)={\rm Cov}(\nabla f(t)),\\
f_{ij}(t)&=\frac{\partial^2 f(t)}{\partial t_it_j}, \quad \nabla^2 f(t)= (f_{ij}(t))_{1\le i, j\le N},
\end{split}
\end{equation*}
and denote by ${\rm index}(\nabla^2 f(t))$ the number of negative eigenvalues of $\nabla^2 f(t)$. We will make use of the following conditions.
\begin{itemize}
\item[({\bf C}1).] $f \in C^2(T)$ almost surely and its second derivatives satisfy the
\emph{mean-square H\"older condition}: for any $t_0\in T$, there exist positive constants $L$, $\eta$ and $\delta$ such that
\begin{equation*}
\E(f_{ij}(t)-f_{ij}(s))^2 \leq L^2 \|t-s\|^{2\eta}, \quad \forall t,s\in U_{t_0}(\delta),\ i, j= 1, \ldots, N.
\end{equation*}

\item[({\bf C}2).]  For every pair $(t, s)\in T^2$ with $t\neq s$, the Gaussian random vector
$$(f(t), \nabla f(t), f_{ij}(t),\,
 f(s), \nabla f(s), f_{ij}(s), 1\leq i\leq j\leq N)$$
is  non-degenerate.
\end{itemize}
Note that $({\bf C}1)$ holds when $f\in C^3(T)$ and $T$ is closed and bounded.

The following theorem, whose proof is given in Section \ref{Section:proofs of main results}, provides the formula for $F_{t_0}(u)$ defined in (\ref{Eq:Palm Ut0}) for smooth Gaussian fields over $\R^N$.
\begin{theorem}\label{Thm:Palm distr} Let $\{f(t): t\in T\}$ be a Gaussian random field satisfying $({\bf C}1)$ and $({\bf C}2)$. Then for each $t_0\in \overset{\circ}{T}$ and $u\in \R$,
\begin{equation}\label{Eq:Palm distr Euclidean}
F_{t_0}(u)= \frac{\E\{|{\rm det} \nabla^2 f(t_0)|\mathbbm{1}_{\{f(t_0)> u\}} \mathbbm{1}_{\{{\rm index}(\nabla^2 f(t_0))=N\}}|\nabla f(t_0)=0\}}{\E\{|{\rm det} \nabla^2 f(t_0)|\mathbbm{1}_{\{{\rm index}(\nabla^2 f(t_0))=N\}} | \nabla f(t_0)=0\}}.
\end{equation}
\end{theorem}

The implicit formula in (\ref{Eq:Palm distr Euclidean}) generalizes the results for stationary Gaussian fields in Cram\'er and Leadbetter (1967, p. 243), Belyaev (1967, 1972) and Lindgren (1972) in the sense that stationarity is no longer required.

Note that the conditional expectations in (\ref{Eq:Palm distr Euclidean}) are hard to compute, since they involve the indicator functions on the eigenvalues of a random matrix. However, in Section \ref{Section:isotropic Euclidean} and Section \ref{Section:isotropic sphere} below, we show that (\ref{Eq:Palm distr Euclidean}) can be computed explicitly for isotropic Gaussian fields.

The following result shows the exact formula for the overshoot distribution defined in (\ref{Eq:overshoot Ut0}).
\begin{theorem}\label{Thm:overshoot distr} Let $\{f(t): t\in T\}$ be a Gaussian random field satisfying $({\bf C}1)$ and $({\bf C}2)$. Then for each $t_0\in \overset{\circ}{T}$, $v\in \R$ and $u>0$,
\begin{equation*}
\begin{split}
\bar{F}_{t_0}(u,v)= \frac{\E\{|{\rm det} \nabla^2 f(t_0)|\mathbbm{1}_{\{f(t_0)> u+v\}} \mathbbm{1}_{\{{\rm index}(\nabla^2 f(t_0))=N\}}|\nabla f(t_0)=0\}}{\E\{|{\rm det} \nabla^2 f(t_0)|\mathbbm{1}_{\{f(t_0)> v\}} \mathbbm{1}_{\{{\rm index}(\nabla^2 f(t_0))=N\}}|\nabla f(t_0)=0\}}.
\end{split}
\end{equation*}
\end{theorem}
\begin{proof}\
The result follows from similar arguments for proving Theorem \ref{Thm:Palm distr}.
\end{proof}

The advantage of overshoot distribution is that we can explore the asymptotics as the pre-threshold $v$ gets large. Theorem \ref{Thm:Palm distr high level} below, whose proof is given in Section \ref{Section:proofs of main results}, provides an asymptotic approximation to the overshoot distribution of a smooth Gaussian field over $\R^N$. This approximation is based on the fact that as the exceeding level tends to infinity, the expected number of local maxima can be approximated by a simpler form which is similar to the expected Euler characteristic of the excursion set.
\begin{theorem}\label{Thm:Palm distr high level} Let $\{f(t): t\in T\}$ be a centered, unit-variance Gaussian random field satisfying $({\bf C}1)$ and $({\bf C}2)$. Then for each $t_0\in \overset{\circ}{T}$ and each fixed $u>0$, there exists $\alpha>0$ such that as $v\to \infty$,
\begin{equation}\label{Eq:Palm distr high level}
\bar{F}_{t_0}(u,v)= \frac{\int_{u+v}^\infty \phi(x) \E\{{\rm det} \nabla^2 f(t_0)|f(t_0)=x, \nabla f(t_0)=0\}dx}{\int_v^\infty \phi(x) \E\{{\rm det} \nabla^2 f(t_0)|f(t_0)=x, \nabla f(t_0)=0\}dx}(1+o(e^{-\alpha v^2})).
\end{equation}
Here and in the sequel, $\phi(x)$ denotes the standard Gaussian density.
\end{theorem}

Note that the expectation in \eqref{Eq:Palm distr high level} is computable since the indicator function does not exist anymore. However, for non-stationary Gaussian random fields over $\R^N$ with $N\ge 2$, the general expression of the expectation in \eqref{Eq:Palm distr high level} would be complicated. Fortunately, as a polynomial in $x$, the coefficient of the highest order of the expectation above is relatively simple, see Lemma \ref{Lem:conditional expectation for N} below. This gives the following approximation to the overshoot distribution for general smooth Gaussian fields over $\R^N$.

\begin{corollary}\label{Cor:Palm distr high level o(1)} Let the assumptions in Theorem \ref{Thm:Palm distr high level} hold. Then for each $t_0\in \overset{\circ}{T}$ and each fixed $u>0$, as $v\to \infty$,
\begin{equation}\label{eq:Palm distr high level o(1)}
\bar{F}_{t_0}(u,v)= \frac{(u+v)^{N-1}e^{-(u+v)^2/2}}{v^{N-1}e^{-v^2/2}}(1+O(v^{-2})).
\end{equation}
\end{corollary}
\begin{proof}\
The result follows immediately from Theorem \ref{Thm:Palm distr high level} and Lemma \ref{Lem:conditional expectation for N} below.
\end{proof}

It can be seen that the result in Corollary \ref{Cor:Palm distr high level o(1)} reduces to the exponential asymptotic distribution given by Belyaev (1967, 1972), Nosko (1969, 1970a, 1970b) and Adler (1981), but the result here gives the approximation error and does not require stationarity. Compared with (\ref{Eq:Palm distr high level}), (\ref{eq:Palm distr high level o(1)}) provides a less accurate approximation, since the error is only $O(v^{-2})$, but it provides a simple explicit form.

Next we show some cases where the approximation in (\ref{Eq:Palm distr high level}) becomes relatively simple and with the same degree of accuracy, i.e., the error is super-exponentially small.
\begin{corollary}\label{Cor:overshoot 1D} Let the assumptions in Theorem \ref{Thm:Palm distr high level} hold. Suppose further the dimension $N=1$ or the field $f$ is stationary,  then for each $t_0\in \overset{\circ}{T}$ and each fixed $u>0$, there exists $\alpha>0$ such that as $v\to \infty$,
\begin{equation}\label{Eq:Palm distr high level stationary}
\bar{F}_{t_0}(u,v)= \frac{H_{N-1}(u+v)e^{-(u+v)^2/2}}{H_{N-1}(v)e^{-v^2/2}}(1+o(e^{-\alpha v^2})),
\end{equation}
where $H_{N-1}(x)$ is the Hermite polynomial of order $N-1$.
\end{corollary}
\begin{proof}\ (i) Suppose first $N=1$. Since ${\rm Var}(f(t)) \equiv 1$, $\E \{f(t)f'(t)\}\equiv 0$ and $\E \{f''(t)f(t)\} = -{\rm Var}(f'(t)) = -\La(t)$. It follows that
\begin{equation*}
\begin{split}
\E &\{{\rm det} \nabla^2 f(t)| f(t)=x, \nabla f(t)=0\} = \E\{f''(t)| f(t)=x, f'(t)=0\}\\
&= (\E \{f''(t)f(t)\}, \E \{f''(t)f'(t)\}) \left( \begin{array}{cc}
1 & 0 \\
0 & \frac{1}{{\rm Var}(f'(t))} \\\end{array} \right) \left( \begin{array}{c}
x \\
0 \\\end{array} \right)= -\La(t) x.
\end{split}
\end{equation*}
Plugging this into (\ref{Eq:Palm distr high level}) yields the desired result.

(ii) If $f$ is stationary, it can be shown that [cf. Lemma 11.7.1 in Adler and Taylor (2007)],
$$
\E\{{\rm det} \nabla^2 f(t_0)| f(t_0)=x, \nabla f(t_0)=0\}=(-1)^N{\rm det}(\La(t_0)) H_N(x).
$$
Then (\ref{Eq:Palm distr high level stationary}) follows from Theorem \ref{Thm:Palm distr high level} and the following formula for Hermite polynomials
\begin{equation*}
\int_v^\infty H_N(x) e^{-x^2/2}\, dx = H_{N-1}(v) e^{-v^2/2}.
\end{equation*}
\end{proof}

An interesting property of the results obtained about the overshoot distribution is that the asymptotic approximations in Corollaries \ref{Cor:Palm distr high level o(1)} and \ref{Cor:overshoot 1D} do not depend on the location $t_0$, even in the case where stationarity is not assumed. In addition, they do not require any knowledge of spectral moments of $f$ except for zero mean and constant variance. In this sense, the distributions are convenient for use in statistics because the correlation function of the field need not be estimated.

\subsection{Isotropic Gaussian Random Fields on Euclidean Space}\label{Section:isotropic Euclidean}
We show here the explicit formulae for both the height distribution and the overshoot distribution of local maxima for isotropic Gaussian random fields. To our knowledge, this article is the first attempt to obtain these distributions explicitly for $N \ge 2$. The main tools are techniques from random matrix theory developed in Fyodorov (2004), Aza\"is and Wschebor (2008) and Auffinger (2011).

Let $\{f(t): t\in T\}$ be a real-valued, $C^2$, centered, unit-variance isotropic Gaussian field parameterized on an $N$-dimensional set $T\subset\R^N$. Due to isotropy, we can write the covariance function of the field as $\E\{f(t)f(s)\}=\rho(\|t-s\|^2)$ for an appropriate function $\rho(\cdot): [0,\infty) \rightarrow \R$, and denote
\begin{equation}\label{Eq:kappa}
\rho'=\rho'(0), \quad \rho''=\rho''(0), \quad \kappa=-\rho'/\sqrt{\rho''}.
\end{equation}
By isotropy again, the covariance of $(f(t), \nabla f(t), \nabla^2 f(t))$ only depends on $\rho'$ and $\rho''$, see Lemma \ref{Lem:cov of isotropic Euclidean} below. In particular, by Lemma \ref{Lem:cov of isotropic Euclidean}, we see that ${\rm Var}(f_i(t))=-2\rho'$ and ${\rm Var}(f_{ii}(t))=12\rho''$ for any $i\in\{1,\ldots, N\}$, which implies $\rho'<0$ and $\rho''>0$ and hence $\kappa>0$. We need the following condition for further discussions.
\begin{itemize}
\item[({\bf C}3).] $\kappa \le 1$ (or equivalently $\rho''-\rho'^2\ge 0$).
\end{itemize}

\begin{example}
\label{Example:rho}
Here are some examples of covariance functions with corresponding $\rho$ satisfying $({\bf C}3)$.

(i) Powered exponential: $\rho(r)=e^{-cr}$, where $c>0$. Then $\rho'=-c$, $\rho''=c^2$ and $\kappa=1$.

(ii) Cauchy: $\rho(r)=(1+r/c)^{-\beta}$, where $c>0$ and $\beta>0$. Then $\rho'=-\beta/c$, $\rho''=\beta(\beta+1)/c^2$ and $\kappa=\sqrt{\beta/(\beta+1)}$.
\end{example}

\begin{remark}\label{Remark:C3}
By Aza\"is and Wschebor (2010), $({\bf C}3)$ holds when $\rho(\|t-s\|^2)$, $t, s\in \R^N$, is a positive definite function for every dimension $N\ge 1$. The cases in Example \ref{Example:rho} are of this kind.
\end{remark}

We shall use \eqref{Eq:Palm distr Euclidean} to compute the distribution of the height of a local maximum. As mentioned before, the conditional distribution on the right hand side of \eqref{Eq:Palm distr Euclidean} is extremely hard to compute. In Section \ref{Section:proofs of main results} below, we build connection between such distribution and certain GOE matrix to make the computation available.

Recall that an $N\times N$ random matrix $M_N$ is said to have the Gaussian Orthogonal Ensemble (GOE) distribution if it is symmetric, with centered Gaussian entries $M_{ij}$ satisfying ${\rm Var}(M_{ii})=1$, ${\rm Var}(M_{ij})=1/2$ if $i<j$ and the random variables $\{M_{ij}, 1\leq i\leq j\leq N\}$ are independent. Moreover, the explicit formula for the distribution $Q_N$ of the eigenvalues $\la_i$ of $M_N$ is given by [cf. Auffinger (2011)]
\begin{equation}\label{Eq:GOE density}
Q_N(d\la)=\frac{1}{c_N} \prod_{i=1}^N e^{-\frac{1}{2}\la_i^2}d\la_i \prod_{1\leq i<j\leq N}|\la_i-\la_j|\mathbbm{1}_{\{\la_1\leq\ldots\leq\la_N\}},
\end{equation}
where the normalization constant $c_N$ can be computed from Selberg's integral
\begin{equation}\label{Eq:normalization constant}
c_N=\frac{1}{N!}(2\sqrt{2})^N \prod_{i=1}^N\Gamma\Big(1+\frac{i}{2}\Big).
\end{equation}
We use notation $\E_{GOE}^N$ to represent the expectation under density $Q_N(d\la)$, i.e., for a measurable function $g$,
$$
\E_{GOE}^N [g(\la_1, \ldots, \la_N)] = \int_{\la_1 \le \ldots \le \la_N} g(\la_1, \ldots, \la_N) Q_N(d\la).
$$

\begin{theorem}\label{Thm:Palm distr iso Euclidean} Let $\{f(t): t\in T\}$ be a centered, unit-variance, isotropic Gaussian random field satisfying $({\bf C}1)$, $({\bf C}2)$ and $({\bf C}3)$. Then for each $t_0\in \overset{\circ}{T}$ and $u\in \R$,
\begin{equation*}
\begin{split}
F_{t_0}(u)= \left\{
  \begin{array}{l l}
     \frac{(1-\kappa^2)^{-1/2} \int_u^\infty \phi(x)\E_{GOE}^{N+1}\left\{ \exp\left[\la_{N+1}^2/2 - (\la_{N+1}-\kappa x/\sqrt{2} )^2/(1-\kappa^2) \right]\right\}dx}{\E_{GOE}^{N+1}\left\{ \exp\left[-\la_{N+1}^2/2 \right] \right\}} & \quad \text{if $\kappa \in (0, 1)$},\\
     \frac{\int_u^\infty \phi(x)\E_{GOE}^{N}\left\{ \left(\prod_{i=1}^N|\la_i-x/\sqrt{2}|\right) \mathbbm{1}_{\{\la_N<x/\sqrt{2}\}} \right\}dx}{\sqrt{2/\pi}\Gamma\left(\frac{N+1}{2}\right)\E_{GOE}^{N+1}\left\{ \exp\left[-\la_{N+1}^2/2 \right] \right\}} & \quad \text{if $\kappa = 1$},
   \end{array} \right.
\end{split}
\end{equation*}
where $\kappa$ is defined in \eqref{Eq:kappa}.
\end{theorem}
\begin{proof}\ Since $f$ is centered and has unit variance, the numerator in \eqref{Thm:Palm distr} can be written as
$$
\int_u^\infty \phi(x) \E\{|{\rm det} \nabla^2 f(t_0)|\mathbbm{1}_{\{{\rm index}(\nabla^2 f(t_0))=N\}}|f(t_0)=x, \nabla f(t_0)=0\}dx.
$$
Applying Theorem \ref{Thm:Palm distr} and Lemmas \ref{Lem:expectation of local max} and \ref{Lem:expectation of local max above u} below gives the desired result.
\end{proof}
\begin{remark}
The formula in Theorem \ref{Thm:Palm distr iso Euclidean} shows that for an isotropic Gaussian field over $\R^N$, $F_{t_0}(u)$ only depends on $\kappa$. Therefore, we may write $F_{t_0}(u)$ as $F_{t_0}(u, \kappa)$. As a consequence of Lemma \ref{Lem:GOE for det Hessian}, $F_{t_0}(u, \kappa)$ is continuous in $\kappa$, hence the formula for the case of $\kappa = 1$ (i.e. $\rho''-\rho'^2= 0$) can also be derived by taking the limit $\lim_{\kappa \uparrow 1}F_{t_0}(u, \kappa)$.
\end{remark}

Next we show an example on computing $F_{t_0}(u)$ explicitly for $N=2$. The calculation for $N=1$ and $N > 2$ is similar and thus omitted here. In particular, the formula for $N=1$ derived in such method can be verified to be the same as in Cramer and Leadbetter (1967).

\begin{example} Let $N=2$. Applying Proposition \ref{Prop:GOE expectation for N=2} below with $a=1$ and $b=0$ gives
\begin{equation}\label{Eq:GOE expectation for a=1 N=2}
\E_{GOE}^{N+1}\bigg\{ \exp\bigg[-\frac{\la_{N+1}^2}{2} \bigg] \bigg\}= \frac{\sqrt{6}}{6}.
\end{equation}
Applying Proposition \ref{Prop:GOE expectation for N=2} again with $a=1/(1-\kappa^2)$ and $b=\kappa x/\sqrt{2}$, one has
\begin{equation}\label{Eq:GOE expectation for a and b N=2}
\begin{split}
&\quad \E_{GOE}^{N+1}\bigg\{ \exp\bigg[\frac{\la_{N+1}^2}{2} - \frac{(\la_{N+1}-\kappa x/\sqrt{2} )^2}{1-\kappa^2} \bigg]\bigg\}\\
&=\frac{\sqrt{1-\kappa^2}}{\pi\sqrt{2}}\bigg\{\pi\kappa^2(x^2-1)\Phi\Big(\frac{\kappa x}{\sqrt{2-\kappa^2}} \Big) + \frac{\kappa x\sqrt{2-\kappa^2}\sqrt{\pi}}{\sqrt{2}}e^{-\frac{\kappa^2x^2}{2(2-\kappa^2)}} \\
&\quad + \frac{2\pi}{\sqrt{3-\kappa^2}}e^{-\frac{\kappa^2x^2}{2(3-\kappa^2)}}\Phi\Big(\frac{\kappa x}{\sqrt{(3-\kappa^2)(2-\kappa^2)}} \Big) \bigg\},
\end{split}
\end{equation}
where $\Phi(x)=(2\pi)^{-1/2}\int_{-\infty}^x e^{-\frac{t^2}{2}} dt$ is the c.d.f. of standard Normal random variable. Let $h(x)$ be the density function of the distribution of the height of a local maximum, i.e. $h(x)=-F'_{t_0}(x)$. By Theorem \ref{Thm:Palm distr iso Euclidean}, (\ref{Eq:GOE expectation for a=1 N=2}) and (\ref{Eq:GOE expectation for a and b N=2}),
\begin{equation}\label{Eq:h on R^2}
\begin{split}
h(x) &=\sqrt{3}\kappa^2(x^2-1)\phi(x)\Phi\Big(\frac{\kappa x}{\sqrt{2-\kappa^2}} \Big) + \frac{\kappa x\sqrt{3(2-\kappa^2)}}{2\pi}e^{-\frac{x^2}{2-\kappa^2}} \\
&\quad+\frac{\sqrt{6}}{\sqrt{\pi(3-\kappa^2)}}e^{-\frac{3x^2}{2(3-\kappa^2)}}\Phi\Big(\frac{\kappa x}{\sqrt{(3-\kappa^2)(2-\kappa^2)}} \Big),
\end{split}
\end{equation}
and hence $F_{t_0}(u)=\int_u^\infty h(x)dx$. Figure \ref{Fig:h on R^2} shows several examples. Shown in solid red is the extreme case of $({\bf C}3)$, $\kappa=1$, which simplifies to
\begin{equation*}
h(x) = \sqrt{3}(x^2-1)\phi(x)\Phi(x) + \frac{\sqrt{3}}{2\pi}xe^{-x^2}+\frac{\sqrt{3}}{\sqrt{\pi}}e^{-\frac{3x^2}{4}}\Phi\Big(\frac{x}{\sqrt{2}} \Big).
\end{equation*}
As an interesting phenomenon, it can be seen from both (\ref{Eq:h on R^2}) and Figure~\ref{Fig:h on R^2} that $h(x) \to \phi(x)$ if $\kappa \to 0$.
\begin{figure}[h!]
  \centering
\includegraphics[scale=0.45]{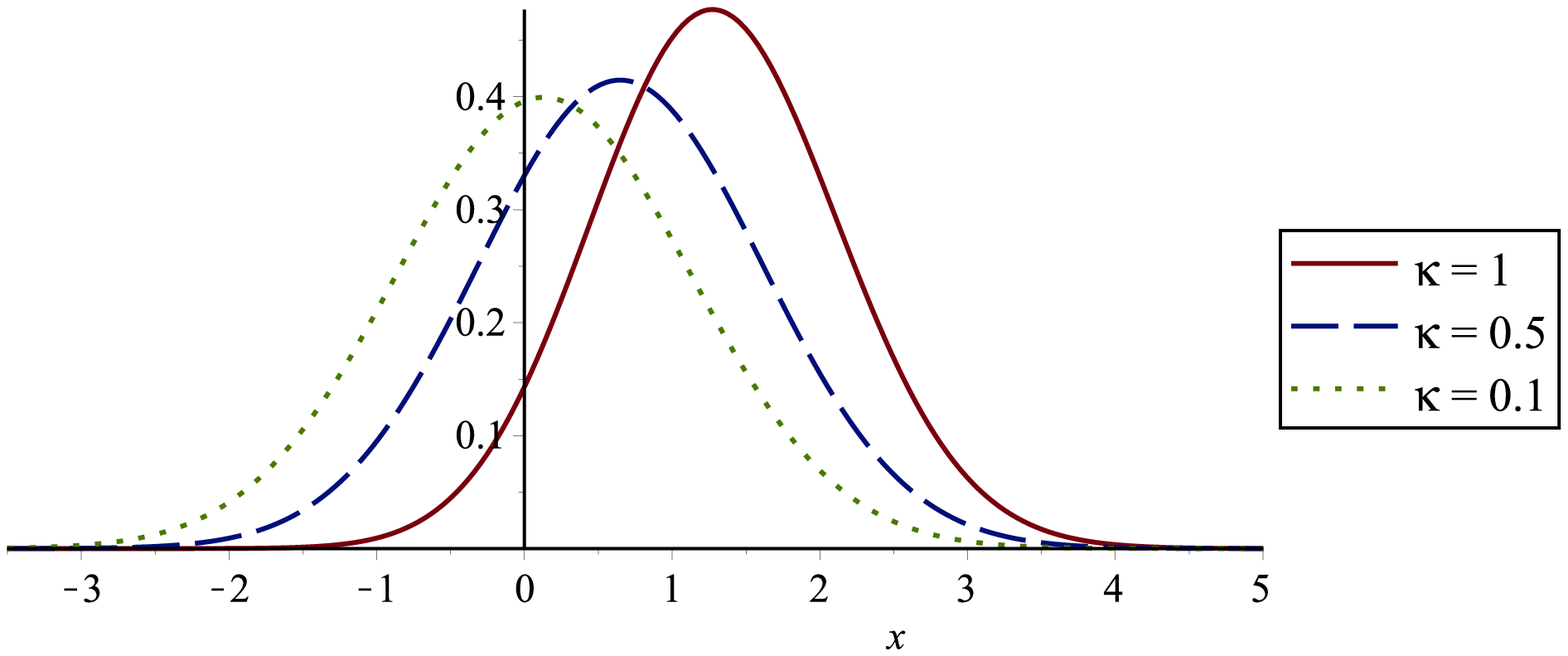}
\caption{Density function $h(x)$ of the distribution $F_{t_0}$ for isotropic Gaussian fields on $\R^2$.}
\label{Fig:h on R^2}
\end{figure}
\end{example}

\begin{theorem}\label{Thm:overshoot isotropic} Let $\{f(t): t\in T\}$ be a centered, unit-variance, isotropic Gaussian random field satisfying $({\bf C}1)$, $({\bf C}2)$ and $({\bf C}3)$. Then for each $t_0\in \overset{\circ}{T}$, $v\in \R$ and $u>0$,
\begin{equation}\label{Eq:overshoot isotropic}
\begin{split}
\bar{F}_{t_0}(u,v)= \left\{
  \begin{array}{l l}
     \frac{\int_{u+v}^\infty \phi(x)\E_{GOE}^{N+1}\left\{ \exp\left[\la_{N+1}^2/2 - (\la_{N+1}-\kappa x/\sqrt{2} )^2/(1-\kappa^2) \right]\right\}dx}{\int_v^\infty \phi(x)\E_{GOE}^{N+1}\left\{ \exp\left[\la_{N+1}^2/2 - (\la_{N+1}-\kappa x/\sqrt{2} )^2/(1-\kappa^2) \right]\right\}dx} & \quad \text{if $\kappa \in (0,1)$},\\
     \frac{\int_{u+v}^\infty \phi(x)\E_{GOE}^{N}\left\{ \left(\prod_{i=1}^N|\la_i-x/\sqrt{2}|\right) \mathbbm{1}_{\{\la_N<x/\sqrt{2}\}} \right\}dx}{\int_v^\infty \phi(x)\E_{GOE}^{N}\left\{ \left(\prod_{i=1}^N|\la_i-x/\sqrt{2}|\right) \mathbbm{1}_{\{\la_N<x/\sqrt{2}\}} \right\}dx} & \quad \text{if $\kappa = 1$},
   \end{array} \right.
\end{split}
\end{equation}
where $\kappa$ is defined in \eqref{Eq:kappa}.
\end{theorem}
\begin{proof}\ The result follows immediately by applying Theorem \ref{Thm:overshoot distr} and Lemma \ref{Lem:expectation of local max above u} below.
\end{proof}

Note that the expectations in \eqref{Eq:overshoot isotropic} can be computed similarly to \eqref{Eq:GOE expectation for a and b N=2} for any $N \ge 1$, thus Theorem \ref{Thm:overshoot isotropic} provides an explicit formula for the overshoot distribution of isotropic Gaussian fields. On the other hand, since isotropy implies stationarity, the approximation to overshoot distribution for large $v$ is simply given by Corollary \ref{Cor:overshoot 1D}.

\section{Smooth Gaussian Random Fields on Manifolds}\label{section:general manifolds}
\subsection{Height Distribution and Overshoot Distribution of Local Maxima}
Let $(M,g)$ be an $N$-dimensional Riemannian manifold and let $f$ be a smooth function on $M$. Then the \emph{gradient} of $f$, denoted by $\nabla f$, is the unique continuous vector field on $M$ such that
$g(\nabla f, X) = X f$
for every vector field $X$. The \emph{Hessian} of $f$, denoted by $\nabla^2 f$, is the double differential form defined by
$\nabla^2 f (X,Y)= XY f - \nabla_X Y f$,
where $X$ and $Y$ are vector fields, $\nabla_X$ is the Levi-Civit\'a connection of $(M,g)$. To make the notations consistent with the Euclidean case, we fix an orthonormal frame $\{E_i\}_{1\le i\le N}$, and let
\begin{equation}\label{Eq:gradient and hessian on manifolds}
\begin{split}
\nabla f &= (f_1, \ldots, f_N)=(E_1f, \ldots, E_Nf),\\
\nabla^2 f &= (f_{ij})_{1\le i, j\le N}= (\nabla^2 f (E_i, E_j))_{1\le i, j\le N}.
\end{split}
\end{equation}
Note that if $t$ is a critical point, i.e. $\nabla f(t)=0$, then $\nabla^2 f (E_i, E_j)(t)=E_iE_jf(t)$, which is similar to the Euclidean case.

Let $B_{t_0}(\ep)=\{t\in M: d(t,t_0)\leq \ep\}$ be the geodesic ball of radius $\ep$ centered at $t_0\in \overset{\circ}{M}$, where $d$ is the distance function induced by the Riemannian metric $g$. We also define $F_{t_0}(u)$ as in (\ref{Eq:Palm Ut0}) and $\bar{F}_{t_0}(u,v)$ as in (\ref{Eq:overshoot Ut0}) with $U_{t_0}(\ep)$ replaced by $B_{t_0}(\ep)$, respectively.

We will make use of the following conditions.
\begin{itemize}
\item[$({\bf C}1')$.] $f \in C^2(M)$ almost surely and its second derivatives satisfy the
\emph{mean-square H\"older condition}: for any $t_0\in M$, there exist positive constants $L$, $\eta$ and $\delta$ such that
\begin{equation*}
\E(f_{ij}(t)-f_{ij}(s))^2 \leq L^2 d(t,s)^{2\eta}, \quad \forall t,s\in B_{t_0}(\delta),\ i, j= 1, \ldots, N.
\end{equation*}

\item[$({\bf C}2')$.]  For every pair $(t, s)\in M^2$ with $t\neq s$, the Gaussian random vector
$$(f(t), \nabla f(t), f_{ij}(t),\,
 f(s), \nabla f(s), f_{ij}(s), 1\leq i\leq j\leq N)$$
is  non-degenerate.
\end{itemize}
Note that $({\bf C}1')$ holds when $f\in C^3(M)$.

Theorem \ref{Thm:Palm distr manifolds} below, whose proof is given in Section \ref{Section:proofs of main results}, is a generalization of Theorems \ref{Thm:Palm distr}, \ref{Thm:overshoot distr}, \ref{Thm:Palm distr high level} and Corollary \ref{Cor:Palm distr high level o(1)}. It provides formulae for both the height distribution and the overshoot distribution of local maxima for smooth Gaussian fields over Riemannian manifolds. Note that the formal expressions are exactly the same as in Euclidean case, but now the field is defined on a manifold.
\begin{theorem}\label{Thm:Palm distr manifolds} Let $(M,g)$ be an oriented $N$-dimensional $C^3$ Riemannian manifold with a $C^1$ Riemannian metric $g$.  Let $f$ be a Gaussian random field on $M$ such that $({\bf C}1')$ and $({\bf C}2')$ are fulfilled. Then for each $t_0 \in \overset{\circ}{M}$, $u, v\in \R$ and $w>0$,
\begin{equation}\label{Eq:Palm distr manifolds 1}
\begin{split}
F_{t_0}(u) &= \frac{\E\{|{\rm det} \nabla^2 f(t_0)|\mathbbm{1}_{\{f(t_0)> u\}} \mathbbm{1}_{\{{\rm index}(\nabla^2 f(t_0))=N\}}|\nabla f(t_0)=0\}}{\E\{|{\rm det} \nabla^2 f(t_0)|\mathbbm{1}_{\{{\rm index}(\nabla^2 f(t_0))=N\}} | \nabla f(t_0)=0\}},\\
\bar{F}_{t_0}(w,v) &= \frac{\E\{|{\rm det} \nabla^2 f(t_0)|\mathbbm{1}_{\{f(t_0)> w+v\}} \mathbbm{1}_{\{{\rm index}(\nabla^2 f(t_0))=N\}}|\nabla f(t_0)=0\}}{\E\{|{\rm det} \nabla^2 f(t_0)|\mathbbm{1}_{\{f(t_0)> v\}} \mathbbm{1}_{\{{\rm index}(\nabla^2 f(t_0))=N\}}|\nabla f(t_0)=0\}}.
\end{split}
\end{equation}
If we assume further that $f$ is centered and has unit variance, then for each fixed $w>0$, there exists $\alpha>0$ such that as $v\to \infty$,
\begin{equation}\label{Eq:Palm distr manifolds 3}
\begin{split}
\bar{F}_{t_0}(w,v) &= \frac{\int_{w+v}^\infty \phi(x) \E\{{\rm det} \nabla^2 f(t_0)|f(t_0)=x, \nabla f(t_0)=0\}dx}{\int_v^\infty \phi(x) \E\{{\rm det} \nabla^2 f(t_0)|f(t_0)=x, \nabla f(t_0)=0\}dx}(1+o(e^{-\alpha v^2}))\\
&= \frac{(w+v)^{N-1}e^{-(w+v)^2/2}}{v^{N-1}e^{-v^2/2}}(1+O(v^{-2})).
\end{split}
\end{equation}
\end{theorem}

It is quite remarkable that the second approximation in \eqref{Eq:Palm distr manifolds 3} does not depend on the curvature of the manifold nor the covariance function of the field, which need not have any stationary properties other than zero mean and constant variance.

\subsection{Isotropic Gaussian Random Fields on the Sphere}\label{Section:isotropic sphere}
Similarly to the Euclidean case, we explore the explicit formulae for both the height distribution and the overshoot distribution of local maxima for isotropic Gaussian random fields on a particular manifold, sphere.

Consider an isotropic Gaussian random field $\{f(t): t\in \mathbb{S}^N\}$, where $\mathbb{S}^N\subset \R^{N+1}$ is the $N$-dimensional unit sphere. For the purpose of simplifying the arguments, we will focus here on the case $N\ge 2$. The special case of the circle, $N=1$, requires separate treatment but extending our results to that case is straightforward.

The following theorem by Schoenberg (1942) characterizes the covariance function of an isotropic Gaussian field on sphere [see also Gneiting (2013)].
\begin{theorem}\label{Thm:Schoenberg} A continuous function $C(\cdot, \cdot):\mathbb{S}^N\times \mathbb{S}^N \rightarrow \R $ is the covariance of an isotropic Gaussian field on $\mathbb{S}^N$, $N \ge 2$, if and only if it has the form
\begin{equation*}
C(t,s)= \sum_{n=0}^\infty a_n P_n^\la(\l t, s \r), \quad  t, s \in \mathbb{S}^N,
\end{equation*}
where $\la=(N-1)/2$, $a_n \geq 0$, $\sum_{n=0}^\infty a_nP_n^\la(1) <\infty$, and $P_n^\la$ are ultraspherical polynomials defined by the expansion
\begin{equation*}
(1-2rx +r^2)^{-\la}=\sum_{n=0}^\infty r^n P_n^\la(x), \quad x\in [-1,1].
\end{equation*}
\end{theorem}

\begin{remark} (i). Note that [cf. Szeg\"o (1975, p. 80)]
\begin{equation}\label{Eq:us polynomials at 1}
P_n^\la(1)=\binom{n+2\la-1}{n}
\end{equation}
and $\la=(N-1)/2$, therefore, $\sum_{n=0}^\infty a_nP_n^\la(1) <\infty$ is equivalent to $\sum_{n=0}^\infty n^{N-2}a_n <\infty$.

(ii). When $N=2$, $\la=1/2$ and $P_n^\la$ become \emph{Legendre polynomials}. For more results on isotropic Gaussian fields on $\mathbb{S}^2$, see a recent monograph by Marinucci and Peccati (2011).

(iii). Theorem \ref{Thm:Schoenberg} still holds for the case $N=1$ if we set [cf. Schoenberg (1942)] $P_n^0(\l t, s \r)=\cos(n\arccos\l t, s \r)=T_n(\l t, s \r)$, where $T_n$ are \emph{Chebyshev polynomials of the first kind} defined by the expansion
\begin{equation*}
\frac{1-rx}{1-2rx +r^2}=\sum_{n=0}^\infty r^n T_n(x), \quad x\in [-1,1].
\end{equation*}
The arguments in the rest of this section can be easily modified accordingly.
\end{remark}

The following statement $({\bf C}1'')$ is a smoothness condition for Gaussian fields on sphere. Lemma \ref{Lem:C^3 sphere} below shows that $({\bf C}1'')$ implies the pervious smoothness condition $({\bf C}1')$.
\begin{itemize}
\item[$({\bf C}1'')$.] The covariance $C(\cdot, \cdot)$ of $\{f(t): t\in \mathbb{S}^N\}$, $N \ge 2$, satisfies
\begin{equation*}
C(t,s)= \sum_{n=0}^\infty a_n P_n^\la(\l t, s \r), \quad  t, s \in \mathbb{S}^N,
\end{equation*}
where $\la=(N-1)/2$, $a_n \geq 0$, $\sum_{n=1}^\infty n^{N+8}a_n<\infty$, and $P_n^\la$ are ultraspherical polynomials.
\end{itemize}

\begin{lemma}\label{Lem:C^3 sphere} {\rm [Cheng and Xiao (2014)].}
Let $f$ be an isotropic Gaussian field on $\mathbb{S}^N$ such that $({\bf C}1'')$ is fulfilled. Then the covariance $C(\cdot, \cdot)\in C^5(\mathbb{S}^N\times \mathbb{S}^N)$ and hence $({\bf C}1')$ holds for $f$.
\end{lemma}

For a unit-variance isotropic Gaussian field $f$ on $\mathbb{S}^N$ satisfying $({\bf C}1'')$, we define
\begin{equation}\label{Def:C' and C''}
\begin{split}
C'&=\sum_{n=1}^\infty a_n\Big(\frac{d}{dx}P_n^\la(x)|_{x=1}\Big)=(N-1)\sum_{n=1}^\infty a_nP_{n-1}^{\la+1}(1), \\
C''&=\sum_{n=2}^\infty a_n\Big(\frac{d^2}{dx^2}P_n^\la(x)|_{x=1}\Big)=(N-1)(N+1)\sum_{n=2}^\infty a_nP_{n-2}^{\la+2}(1).
\end{split}
\end{equation}
Due to isotropy, the covariance of $(f(t), \nabla f(t), \nabla^2 f(t))$ only depends on $C'$ and $C''$, see Lemma \ref{Lem:joint distribution sphere} below. In particular, by Lemma \ref{Lem:joint distribution sphere} again, ${\rm Var}(f_i(t))=C'$ and ${\rm Var}(f_{ii}(t))=C'+3C''$ for any $i\in\{1,\ldots, N\}$. We need the following condition on $C'$ and $C''$ for further discussions.
\begin{itemize}
\item[$({\bf C}3')$.] $C''+C'-C'^2\ge 0$.
\end{itemize}
\begin{remark}\label{Remark:C3'}\
Note that $({\bf C}3')$ holds when $C(\cdot, \cdot)$ is a covariance function (i.e. positive definite function) for every dimension $N\ge 2$ (or equivalently for every $N\ge 1$). In fact, by Schoenberg (1942), if $C(\cdot, \cdot)$ is a covariance function on $\mathbb{S}^N$ for every $N\ge 2$, then it is necessary of the form
\begin{equation*}
C(t,s)= \sum_{n=0}^\infty b_n \l t, s \r^n, \quad  t, s \in \mathbb{S}^N,
\end{equation*}
where $b_n\ge 0$. Unit-variance of the field implies $\sum_{n=0}^\infty b_n=1$. Now consider the random variable $X$ that assigns probability $b_n$ to the integer $n$. Then $C'=\sum_{n=1}^\infty nb_n=\E X$, $C''=\sum_{n=2}^\infty n(n-1)b_n=\E X(X-1)$ and $C''+C'-C'^2={\rm Var}(X)\ge 0$, hence $({\bf C}3')$ holds.
\end{remark}

\begin{theorem}\label{Thm:Palm distr sphere} Let $\{f(t): t\in \mathbb{S}^N\}$, $N\ge 2$, be a centered, unit-variance, isotropic Gaussian field satisfying $({\bf C}1'')$, $({\bf C}2')$ and $({\bf C}3')$. Then for each $t_0\in \mathbb{S}^N$ and $u\in \R$,
\begin{equation*}
\begin{split}
F_{t_0}(u)= \left\{
  \begin{array}{l l}
     \frac{\big(\frac{C''+C'}{C''+C'-C'^2}\big)^{1/2}\int_u^\infty \phi(x)\E_{GOE}^{N+1}\Big\{ \exp\Big[\frac{\la_{N+1}^2}{2} - \frac{C''\big(\la_{N+1}-\frac{C'x}{\sqrt{2C''}} \big)^2}{C''+C'-C'^2} \Big]\Big\}dx}{\E_{GOE}^{N+1}\big\{ \exp\big[\frac{1}{2}\la_{N+1}^2 -\frac{C''}{C''+C'}\la_{N+1}^2 \big] \big\}} & \text{if $C''+C'-C'^2> 0$},\\
     \frac{\int_u^\infty \phi(x)\E_{GOE}^{N}\big\{ \big(\prod_{i=1}^N|\la_i-\frac{C'x}{\sqrt{2C''}}|\big) \mathbbm{1}_{\{\la_N<\frac{C'x}{\sqrt{2C''}}\}} \big\} dx}{\big(\frac{2C''}{\pi(C''+C')}\big)^{1/2} \Gamma\left(\frac{N+1}{2}\right) \E_{GOE}^{N+1}\big\{ \exp\big[\frac{1}{2}\la_{N+1}^2 -\frac{C''}{C''+C'}\la_{N+1}^2 \big] \big\}} & \text{if $C''+C'-C'^2= 0$},
   \end{array} \right.
\end{split}
\end{equation*}
where $C'$ and $C''$ are defined in (\ref{Def:C' and C''}).
\end{theorem}
\begin{proof}\ The result follows from applying Theorem \ref{Thm:Palm distr manifolds} and Lemmas \ref{Lem:expectation of local max sphere} and \ref{Lem:expectation of local max above u sphere} below.
\end{proof}

\begin{remark}
The formula in Theorem \ref{Thm:Palm distr sphere} shows that for isotropic Gaussian fields over $\mathbb{S}^N$, $F_{t_0}(u)$ depends on both $C'$ and $C''$. Therefore, we may write $F_{t_0}(u)$ as $F_{t_0}(u, C', C'')$. As a consequence of Lemma \ref{Lem:GOE for det Hessian sphere}, $F_{t_0}(u, C', C'')$ is continuous in $C'$ and $C''$, hence the formula for the case of $C''+C'-C'^2= 0$ can also be derived by taking the limit $\lim_{C''+C'-C'^2 \downarrow 0}F_{t_0}(u, C', C'')$.
\end{remark}

\begin{example} Let $N=2$. Applying Proposition \ref{Prop:GOE expectation for N=2} with $a=\frac{C''}{C''+C'}$ and $b=0$ gives
\begin{equation}\label{Eq:GOE expectation for a=1 N=2 sphere}
\E_{GOE}^{N+1}\bigg\{ \exp\bigg[\frac{1}{2}\la_{N+1}^2 -\frac{C''}{C''+C'}\la_{N+1}^2 \bigg] \bigg\}= \frac{\sqrt{2}}{2}\bigg\{\frac{C'}{2C''}\Big( \frac{C''+C'}{C''}\Big)^{1/2}+\Big( \frac{C''+C'}{3C''+C'}\Big)^{1/2}\bigg\}.
\end{equation}
Applying Proposition \ref{Prop:GOE expectation for N=2} again with $a=\frac{C''}{C''+C'-C'^2}$ and $b=\frac{C'x}{\sqrt{2C''}}$, one has
\begin{equation}\label{Eq:GOE expectation for a and b N=2 sphere}
\begin{split}
&\quad \E_{GOE}^{N+1}\bigg\{ \exp\bigg[\frac{\la_{N+1}^2}{2} - \frac{C''\Big(\la_{N+1}-\frac{C'x}{\sqrt{2C''}} \Big)^2}{C''+C'-C'^2} \bigg]\bigg\}\\
&=\frac{1}{\pi\sqrt{2}}\Big( \frac{C''+C'-C'^2}{C''} \Big)^{1/2}\bigg\{ \frac{C'^2(x^2-1)+C'}{C''}\pi\Phi\Big(\frac{C'x}{\sqrt{2C''+C'-C'^2}} \Big) \\
&\quad+ \frac{xC'\sqrt{2C''+C'-C'^2}}{C''\sqrt{2}}\sqrt{\pi}e^{-\frac{C'^2x^2}{2(2C''+C'-C'^2)}} \\
&\quad + \frac{2\pi\sqrt{C''}}{\sqrt{3C''+C'-C'^2}}e^{-\frac{C'^2x^2}{2(3C''+C'-C'^2)}}\Phi\Big(\frac{xC'\sqrt{C''}}{\sqrt{(2C''+C'-C'^2)(3C''+C'-C'^2)}} \Big) \bigg\}.
\end{split}
\end{equation}
Let $h(x)$ be the density function of the distribution of the height of a local maximum, i.e. $h(x)=-F'_{t_0}(x)$. By Theorem \ref{Thm:Palm distr sphere}, together with (\ref{Eq:GOE expectation for a=1 N=2 sphere}) and (\ref{Eq:GOE expectation for a and b N=2 sphere}), we obtain
\begin{equation}\label{Eq:h on S^2}
\begin{split}
h(x)&=\bigg(\frac{C'}{2C''}+\Big(\frac{C''}{3C''+C'}\Big)^{1/2}\bigg)^{-1} \bigg\{ \frac{C'^2(x^2-1)+C'}{C''}\phi(x)\Phi\Big(\frac{C'x}{\sqrt{2C''+C'-C'^2}} \Big) \\
&\quad+ \frac{xC'\sqrt{2C''+C'-C'^2}}{2\pi C''}e^{-\frac{(2C''+C')x^2}{2(2C''+C'-C'^2)}} \\
&\quad + \frac{\sqrt{2C''}}{\sqrt{\pi}\sqrt{3C''+C'-C'^2}}e^{-\frac{(3C''+C')x^2}{2(3C''+C'-C'^2)}} \Phi\Big(\frac{xC'\sqrt{C''}}{\sqrt{(2C''+C'-C'^2)(3C''+C'-C'^2)}} \Big) \bigg\},
\end{split}
\end{equation}
and hence $F_{t_0}(u)=\int_u^\infty h(x)dx$.
Figure \ref{Fig:h on S^2} shows several examples. The extreme case of $({\bf C}3')$, $C''+C'-C'^2=0$, is obtained when $C(t,s) = \l t, s \r^n$, $n \ge 2$. Shown in solid red is the case $n=2$, which simplifies to
\[
h(x) = (2x^2-1)\phi(x)\Phi(\sqrt{2}x) + \frac{x \sqrt{2}}{2\pi} e^{-\frac{3x^2}{2}} + \frac{1}{\sqrt{\pi}}e^{-x^2} \Phi(x).
\]
\begin{figure}[h!]
  \centering
\includegraphics[scale=0.5]{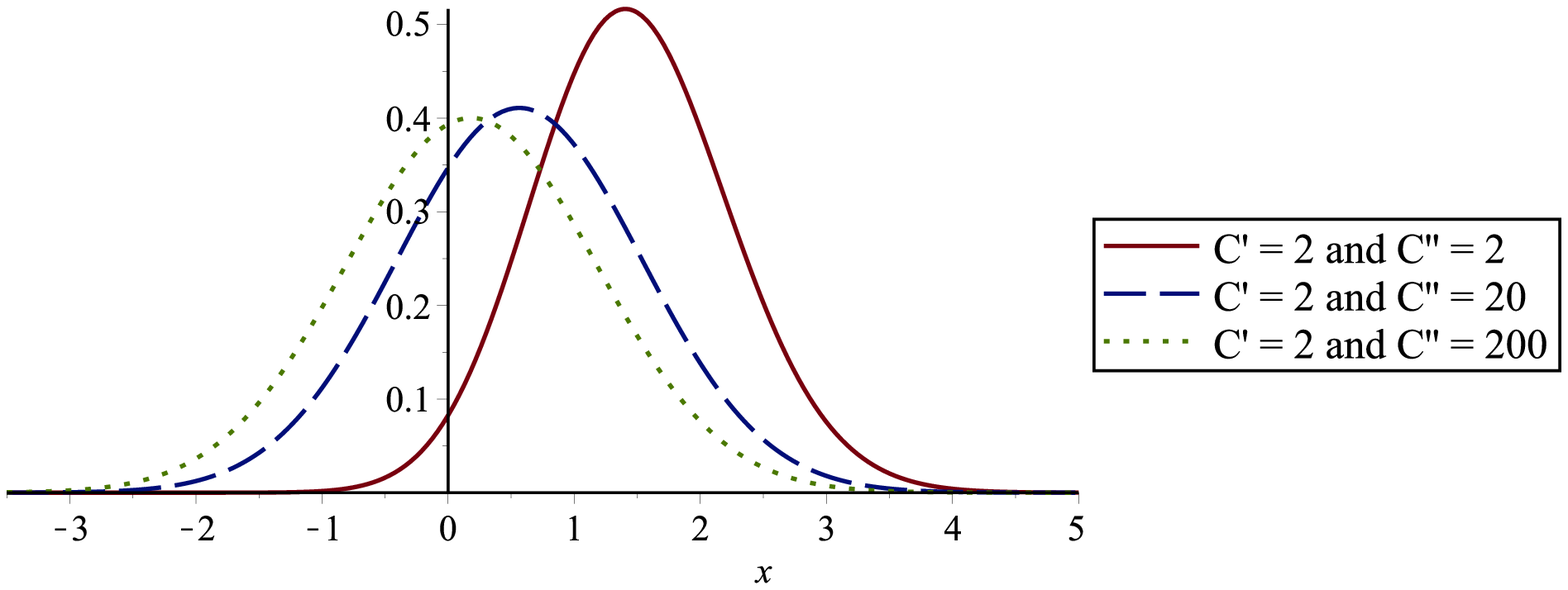}
\caption{Density function $h(x)$ of the distribution $F_{t_0}$ for isotropic Gaussian fields on $\mathbb{S}^2$.}
\label{Fig:h on S^2}
\end{figure}
It can be seen from both (\ref{Eq:h on S^2}) and Figure~\ref{Fig:h on S^2} that $h(x) \to \phi(x)$ if $ \max (C', C'^2)/C'' \to 0$.
\end{example}

\begin{theorem}\label{Thm:Overshoot sphere exact}
 Let $\{f(t): t\in \mathbb{S}^N\}$, $N\ge 2$, be a centered, unit-variance, isotropic Gaussian field satisfying $({\bf C}1'')$, $({\bf C}2')$ and $({\bf C}3')$. Then for each $t_0\in \mathbb{S}^N$ and $u, v>0$,
\begin{equation*}
\begin{split}
\bar{F}_{t_0}(u,v)= \left\{
  \begin{array}{l l}
     \frac{\int_{u+v}^\infty \phi(x)\E_{GOE}^{N+1}\Big\{ \exp\Big[\frac{\la_{N+1}^2}{2} - \frac{C''\big(\la_{N+1}-\frac{C'x}{\sqrt{2C''}} \big)^2}{C''+C'-C'^2} \Big]\Big\}dx}{\int_{v}^\infty \phi(x)\E_{GOE}^{N+1}\Big\{ \exp\Big[\frac{\la_{N+1}^2}{2} - \frac{C''\big(\la_{N+1}-\frac{C'x}{\sqrt{2C''}} \big)^2}{C''+C'-C'^2} \Big]\Big\}dx} & \quad \text{if $C''+C'-C'^2> 0$},\\
     \frac{\int_{u+v}^\infty \phi(x)\E_{GOE}^{N}\big\{ \big(\prod_{i=1}^N|\la_i-\frac{C'x}{\sqrt{2C''}}|\big) \mathbbm{1}_{\{\la_N<\frac{C'x}{\sqrt{2C''}}\}} \big\} dx}{\int_v^\infty \phi(x)\E_{GOE}^{N}\big\{ \big(\prod_{i=1}^N|\la_i-\frac{C'x}{\sqrt{2C''}}|\big) \mathbbm{1}_{\{\la_N<\frac{C'x}{\sqrt{2C''}}\}} \big\} dx} & \quad \text{if $C''+C'-C'^2= 0$},
   \end{array} \right.
\end{split}
\end{equation*}
where $C'$ and $C''$ are defined in (\ref{Def:C' and C''}).
\end{theorem}
\begin{proof}\ The result follows immediately by applying Theorem \ref{Thm:Palm distr manifolds} and Lemma \ref{Lem:expectation of local max above u sphere}.
\end{proof}

Because the exact expression in Theorem \ref{Thm:Overshoot sphere exact} may be complicated for large $N$, we now derive a tight approximation to it, which is analogous to Corollary \ref{Cor:overshoot 1D} for the Euclidean case.

Let $\chi(A_u(f,\mathbb{S}^N))$ be the Euler characteristic of the excursion set $A_u(f,\mathbb{S}^N) = \{t\in \mathbb{S}^N: f(t)> u\}$. Let $\omega_j = {\rm Vol}(\mathbb{S}^j)$, the spherical area of the $j$-dimensional unit sphere $\mathbb{S}^j$, i.e., $\omega_j=2\pi^{(j+1)/2}/\Gamma(\frac{j+1}{2})$. The lemma below provides the formula for the expected Euler characteristic of the excursion set.
\begin{lemma}\label{Lem:MEC sphere} {\rm [Cheng and Xiao (2014)].} Let $\{f(t): t\in \mathbb{S}^N\}$, $N\ge 2$, be a centered, unit-variance, isotropic Gaussian field satisfying $({\bf C}1'')$ and $({\bf C}2')$. Then
\begin{equation*}
\begin{split}
\E\{\chi(A_u(f,\mathbb{S}^N))\} = \sum_{j=0}^N  (C')^{j/2} \mathcal{L}_j (\mathbb{S}^N) \rho_j(u),
\end{split}
\end{equation*}
where $C'$ is defined in (\ref{Def:C' and C''}), $\rho_0(u)= 1-\Phi(u)$, $\rho_j(u) = (2\pi)^{-(j+1)/2} H_{j-1} (u) e^{-u^2/2}$ with Hermite polynomials $H_{j-1}$ for $j\geq 1$ and, for $j=0, \ldots, N$,
\begin{equation}\label{Eq:L-K curvature}
\begin{split}
\mathcal{L}_j (\mathbb{S}^N) = \left\{
  \begin{array}{l l}
     2 \binom{N}{j}\frac{\omega_N}{\omega_{N-j}} & \quad \text{if $N-j$ is even}\\
     0 & \quad \text{otherwise}
   \end{array} \right.
\end{split}
\end{equation}
are the Lipschitz-Killing curvatures of $\mathbb{S}^N$.
\end{lemma}

\begin{theorem}\label{Thm:overshoot sphere} Let $\{f(t): t\in \mathbb{S}^N\}$, $N\ge 2$, be a centered, unit-variance, isotropic Gaussian field satisfying $({\bf C}1'')$ and $({\bf C}2')$. Then for each $t_0\in \mathbb{S}^N$ and each fixed $u>0$, there exists $\alpha>0$ such that as $v\to \infty$,
\begin{equation}\label{Eq:overshoot sphere}
\begin{split}
\bar{F}_{t_0}(u,v)=\frac{\sum_{j=0}^N  (C')^{j/2} \mathcal{L}_j (\mathbb{S}^N) \rho_j(u+v)}{\sum_{j=0}^N  (C')^{j/2} \mathcal{L}_j (\mathbb{S}^N) \rho_j(v)}(1+o(e^{-\alpha v^2})),
\end{split}
\end{equation}
where $C'$ is defined in (\ref{Def:C' and C''}), $\rho_j(u)$ and $\mathcal{L}_j (\mathbb{S}^N)$ are as in Lemma \ref{Lem:MEC sphere}.
\end{theorem}

\begin{remark}
Note that \eqref{Eq:overshoot sphere} depends on the covariance function only through its first derivative $C'$. In comparison with Corollary \ref{Cor:overshoot 1D} for the Euclidean case, there we only have the highest order term of the expected Euler characteristic expansion because we do not consider the boundaries of $T$. On the sphere, we need all terms in the expansion since sphere has no boundary.
\end{remark}

\begin{proof}\ By Theorem \ref{Thm:Palm distr manifolds},
\begin{equation*}
\bar{F}_{t_0}(u,v)= \frac{\int_{u+v}^\infty \phi(x) \E\{{\rm det} \nabla^2 f(t_0)|f(t_0)=x, \nabla f(t_0)=0\}dx}{\int_v^\infty \phi(x) \E\{{\rm det} \nabla^2 f(t_0)|f(t_0)=x, \nabla f(t_0)=0\}dx}(1+o(e^{-\alpha v^2})).
\end{equation*}
Since $f$ is isotropic, integrating the numerator and denominator above over $\mathbb{S}^N$, we obtain
\begin{equation*}
\begin{split}
\bar{F}_{t_0}(u,v)&= \frac{\int_{\mathbb{S}^N}\int_{u+v}^\infty \phi(x) \E\{{\rm det} \nabla^2 f(t)|f(t)=x, \nabla f(t)=0\}dx dt}{\int_{\mathbb{S}^N} \int_v^\infty \phi(x) \E\{{\rm det} \nabla^2 f(t)|f(t)=x, \nabla f(t)=0\}dxdt}(1+o(e^{-\alpha v^2}))\\
&=\frac{\E\{\chi(A_{u+v}(f,\mathbb{S}^N))\}}{\E\{\chi(A_v(f,\mathbb{S}^N))\}}(1+o(e^{-\alpha v^2})),
\end{split}
\end{equation*}
where the last line comes from applying the Kac-Rice Metatheorem to the Euler characteristic of the excursion set, see Adler and Taylor (2007, pp. 315-316).  The result then follows from Lemma \ref{Lem:MEC sphere}.
\end{proof}

\section{Proofs and Auxiliary Results}\label{Section:proofs of main results}
\subsection{Proofs for Section \ref{Section:general Euclidean}}
For $u>0$, let $\mu(t_0, \ep)$, $\mu_N(t_0, \ep)$, $\mu_N^u(t_0, \ep)$ and $\mu_N^{u-}(t_0, \ep)$ be the number of critical points, the number of local maxima, the number of local maxima above $u$ and the number of local maxima below $u$ in $U_{t_0}(\ep)$ respectively. More precisely,
\begin{equation}\label{Def:various mu's}
\begin{split}
\mu(t_0, \ep) &=\#\{t\in U_{t_0}(\ep): \nabla f(t)=0\},\\
\mu_N(t_0, \ep) &=\#\{t\in U_{t_0}(\ep): \nabla f(t)=0, {\rm index}(\nabla^2 f(t))=N\},\\
\mu_N^u(t_0, \ep) &=\#\{t\in U_{t_0}(\ep): f(t)> u, \nabla f(t)=0, {\rm index}(\nabla^2 f(t))=N\},\\
\mu_N^{u-}(t_0, \ep) &=\#\{t\in U_{t_0}(\ep): f(t)\le u, \nabla f(t)=0, {\rm index}(\nabla^2 f(t))=N\},
\end{split}
\end{equation}
where ${\rm index}(\nabla^2 f(t))$ is the number of negative eigenvalues of $\nabla^2 f(t)$.

In order to prove Theorem \ref{Thm:Palm distr}, we need the following lemma which shows that, for the number of critical points over the cube of lengh $\ep$, its factorial moment decays faster than the expectation as $\ep$ tends to 0. Our proof is based on similar arguments in the proof of Lemma 3 in Piterbarg (1996).
\begin{lemma}\label{Lem:Piterbarg} Let $\{f(t): t\in T\}$ be a Gaussian random field satisfying $({\bf C}1)$ and $({\bf C}2)$. Then for each fixed $t_0\in \overset{\circ}{T}$, as $\ep\to 0$,
\begin{equation*}
\E \{\mu(t_0, \ep)(\mu(t_0, \ep) -1)\}= o(\ep^N).
\end{equation*}
\end{lemma}
\begin{proof}\ By the Kac-Rice formula for factorial moments [cf. Theorem 11.5.1 in Adler and Taylor (2007)],
\begin{equation}\label{Eq:factorial moment}
\begin{split}
\E\{\mu(t_0, \ep)(\mu(t_0, \ep) -1)\}=\int_{U_{t_0}(\ep)}\int_{U_{t_0}(\ep)}  E_1(t,s) p_{\nabla f(t), \nabla f(s)}(0,0) dtds,
\end{split}
\end{equation}
where
$$
E_1(t,s) = \E\{|{\rm det}\nabla^2 f(t)||{\rm det}\nabla^2 f(s)|| \nabla f(t)=\nabla f(s)=0\}.
$$
By Taylor's expansion,
\begin{equation}\label{Eq:Taylor expansion}
\nabla f(s)= \nabla f(t) + \nabla^2 f(t)(s-t)^T + \|s-t\|^{1+\eta}\mathbf{Z}_{t,s},
\end{equation}
where $\mathbf{Z}_{t,s}=(Z_{t,s}^1, \ldots, Z_{t,s}^N)^T$ is a Gaussian vector field, with properties to be specified. In particular, by condition ({\bf C}1), for $\ep$ small enough,
\begin{equation*}
\sup_{t,s\in U_{t_0}(\ep), \, t\ne s}\E\|\mathbf{Z}_{t,s}\|^2 \leq C_1,
\end{equation*}
where $C_1$ is some positive constant. Therefore, we can write
$$
E_1(t,s) = \E\{|{\rm det}\nabla^2 f(t)||{\rm det}\nabla^2 f(s)|| \nabla f(t)=0, \nabla^2 f(t)(s-t)^T =- \|s-t\|^{1+\eta}\mathbf{Z}_{t,s}\}.
$$
Note that the determinant of the matrix $\nabla^2 f(t)$ is equal to the determinant of the matrix
\begin{equation*}
\begin{split}
\left(
\begin{array}{cccc}
1 & -(s_1-t_1) & \cdots & -(s_N-t_N) \\
0 &  &  & \\
\vdots & & \nabla^2 f(t) &\\
0 &  &  &
\end{array}
\right).
\end{split}
\end{equation*}
For any $i=2,\ldots, N+1$, multiply the $i$th column of this matrix by $(s_i-t_i)/\|s_i-t_i\|^2$, take the sum of all such columns and add the result to the first column, obtaining the matrix
\begin{equation*}
\begin{split}
\left(
\begin{array}{cccc}
0 & -(s_1-t_1) & \cdots & -(s_N-t_N) \\
-\|s-t\|^{-1+\eta}Z_{t,s}^1 &  &  & \\
\vdots & & \nabla^2 f(t) &\\
-\|s-t\|^{-1+\eta}Z_{t,s}^N &  &  &
\end{array}
\right),
\end{split}
\end{equation*}
whose determinant is still equal to the determinant of $\nabla^2 f(t)$. Let $r=\max_{1\leq i\leq N}|s_i-t_i|$,
\begin{equation*}
\begin{split}
A_{t,s}=\left(
\begin{array}{cccc}
0 & -(s_1-t_1)/r & \cdots & -(s_N-t_N)/r \\
Z_{t,s}^1 &  &  & \\
\vdots & & \nabla^2 f(t) &\\
Z_{t,s}^N &  &  &
\end{array}
\right).
\end{split}
\end{equation*}
Using properties of a determinant, it follows that
\begin{equation*}
|{\rm det} \nabla^2 f(t)|= r\|s-t\|^{-1+\eta}|{\rm det}A_{t,s}|\leq \|s-t\|^{\eta}|{\rm det}A_{t,s}|.
\end{equation*}
Let $e_{t,s}=(s-t)^T/\|s-t\|$, then we obtain
\begin{equation}\label{Eq:E1 and E2}
E_1(t,s)\leq \|s-t\|^{\eta}E_2(t,s),
\end{equation}
where
\begin{equation*}
\begin{split}
E_2(t,s) &= \E\{|{\rm det}A_{t,s}||{\rm det}\nabla^2 f(s)|| \nabla f(t)=0, \nabla^2 f(t)(s-t)^T =- \|s-t\|^{1+\eta}\mathbf{Z}_{t,s}\}\\
&= \E\{|{\rm det}A_{t,s}||{\rm det}\nabla^2 f(s)|| \nabla f(t)=0, \nabla^2 f(t)e_{t,s} + \|s-t\|^\eta\mathbf{Z}_{t,s} =0\}.
\end{split}
\end{equation*}
By ({\bf C}1) and ({\bf C}2), there exists $C_2>0$ such that
\begin{equation*}
\sup_{t,s\in U_{t_0}(\ep), \, t\ne s} E_2(t,s) \le C_2.
\end{equation*}
By (\ref{Eq:factorial moment}) and (\ref{Eq:E1 and E2}),
\begin{equation*}
\begin{split}
\E\{\mu(t_0, \ep)(\mu(t_0, \ep) -1)\}\le C_2\int_{U_{t_0}(\ep)}\int_{U_{t_0}(\ep)} \|s-t\|^{\eta} p_{\nabla f(t), \nabla f(s)}(0,0) dtds.
\end{split}
\end{equation*}

It is obvious that
\begin{equation*}
p_{\nabla f(t), \nabla f(s)}(0,0) \le \frac{1}{(2\pi)^N \sqrt{{\rm det Cov}(\nabla f(t), \nabla f(s))}}.
\end{equation*}
Applying Taylor's expansion (\ref{Eq:Taylor expansion}), we obtain that as $\|s-t\| \to 0$,
\begin{equation*}
\begin{split}
&{\rm det Cov}(\nabla f(t), \nabla f(s)) \\
&\quad = {\rm det Cov}(\nabla f(t), \nabla f(t) + \nabla^2 f(t)(s-t)^T + \|s-t\|^{1+\eta}\mathbf{Z}_{t,s})\\
&\quad = {\rm det Cov}(\nabla f(t), \nabla^2 f(t)(s-t)^T + \|s-t\|^{1+\eta}\mathbf{Z}_{t,s})\\
&\quad = \|s-t\|^{2N} {\rm det Cov}(\nabla f(t), \nabla^2 f(t)e_{t,s} + \|s-t\|^\eta\mathbf{Z}_{t,s})\\
&\quad = \|s-t\|^{2N} {\rm det Cov}(\nabla f(t), \nabla^2 f(t)e_{t,s})(1+o(1)),
\end{split}
\end{equation*}
where the last determinant is bounded away from zero uniformly in $t$ and $s$ due to the regularity condition ({\bf C}2). Therefore, there exists $C_3>0$ such that
\begin{equation*}
\begin{split}
\E \{\mu(t_0, \ep)(\mu(t_0, \ep) -1)\} \leq C_3\int_{U_{t_0}(\ep)}\int_{U_{t_0}(\ep)}\frac{1}{\|t-s\|^{N-\eta}}dtds,
\end{split}
\end{equation*}
where $C_3$ and $\eta$ are some positive constants. Recall the elementary inequality
$$
\frac{x_1+\cdots +x_N}{N} \geq (x_1\cdots x_N)^{1/N}, \quad \forall x_1, \ldots, x_N>0.
$$
It follows that
\begin{equation*}
\begin{split}
\E \{\mu(t_0, \ep)(\mu(t_0, \ep) -1)\} &\leq C_3N^{\eta-N}\int_{U_{t_0}(\ep)}\int_{U_{t_0}(\ep)}\prod_{i=1}^N |t_i-s_i|^{\frac{\eta}{N}-1}dtds\\ &= C_3N^{\eta-N} \bigg(\int_{-\ep/2}^{\ep/2} \int_{-\ep/2}^{\ep/2} |x-y|^{\frac{\eta}{N}-1}dxdy\bigg)^N \\
&=  C_3N^\eta\bigg(\frac{2N}{\eta(\eta+N)} \bigg)^N \ep^{N+\eta} = o(\ep^N).
\end{split}
\end{equation*}
\end{proof}

\begin{proof}{\bf of Theorem \ref{Thm:Palm distr}}\ By the definition in (\ref{Eq:Palm Ut0}),
\begin{equation}\label{Eq:conditional prob form of Palm}
\begin{split}
F_{t_0}(u) =\lim_{\ep\to 0} \frac{ \P\{f(t_0)>u, \mu_N(t_0, \ep)\geq 1\}}{\P\{ \mu_N(t_0, \ep)\geq 1\}}.
\end{split}
\end{equation}
Let $p_i=\P\{\mu_N(t_0, \ep)=i\}$, then $\P\{\mu_N(t_0, \ep) \geq 1 \} = \sum_{i=1}^\infty p_i$ and $\E \{\mu_N(t_0, \ep) \} = \sum_{i=1}^\infty ip_i$, it follows that
\begin{equation*}
\begin{split}
\E \{\mu_N(t_0, \ep) \} - \P\{\mu_N(t_0, \ep) \geq 1 \} &= \sum_{i=2}^\infty (i-1)p_i \\
&\leq \sum_{i=2}^\infty \frac{i(i-1)}{2}p_i = \frac{1}{2}\E \{\mu_N(t_0, \ep)(\mu_N(t_0, \ep) -1)\}.
\end{split}
\end{equation*}
Therefore, by Lemma \ref{Lem:Piterbarg}, as $\ep\to 0$,
\begin{equation}\label{Eq:estimate prob of local max}
\begin{split}
\P\{\mu_N(t_0, \ep) \geq 1 \} = \E \{\mu_N(t_0, \ep) \} + o(\ep^N).
\end{split}
\end{equation}
Similarly,
\begin{equation}\label{Eq:estimate prob of local max above u}
\begin{split}
\P\{\mu_N^u(t_0, \ep) \geq 1 \} = \E \{\mu_N^u(t_0, \ep) \} + o(\ep^N).
\end{split}
\end{equation}
Next we show that
\begin{equation}\label{Eq:difference in numerator}
\begin{split}
|\P\{f(t_0)>u, \mu_N(t_0, \ep)\geq 1\}-\P\{\mu_N^u(t_0, \ep) \geq 1 \}| = o(\ep^N).
\end{split}
\end{equation}
Roughly speaking, the probability that there exists a local maximum and the field exceeds $u$ at $t_0$ is approximately the same as the probability that there is at least one local maximum exceeding u. This is because in the limit, the local maximum occurs at $t_0$ and is greater than $u$. We show the rigorous proof below.

Note that for any evens $A$, $B$, $C$ such that $C\subset B$,
\begin{equation*}
|\P(AB)-\P(C)| \leq \P(ABC^c) + \P(A^cC).
\end{equation*}
By this inequality, to prove (\ref{Eq:difference in numerator}), it suffices to show
\begin{equation*}
\P\{f(t_0)>u, \mu_N(t_0, \ep)\geq 1, \mu_N^u(t_0, \ep) = 0\} + \P\{f(t_0)\leq u, \mu_N^u(t_0, \ep) \geq 1 \} = o(\ep^N),
\end{equation*}
where the first probability above is the probability that the field exceeds $u$ at $t_0$ but all local maxima are below $u$, while the second one is the probability that the field does not exceed $u$ at $t_0$ but all local maxima exceed $u$.

Recall the definition of $\mu_N^{u-}(t_0, \ep)$ in (\ref{Def:various mu's}), we have
\begin{equation}\label{Eq:the first small prob}
\begin{split}
\P\{f(t_0)>u, \mu_N(t_0, \ep)\geq 1, \mu_N^u(t_0, \ep) = 0\} &\leq \P\{f(t_0)>u, \mu_N^{u-}(t_0, \ep)\geq 1\}\\
&= \E \{\mu_N^{u-}(t_0, \ep)\mathbbm{1}_{\{f(t_0)>u\}} \} + o(\ep^N),
\end{split}
\end{equation}
where the second line follows from similar argument for showing (\ref{Eq:estimate prob of local max}). By the Kac-Rice metatheorem,
\begin{equation}\label{Eq:conditional expectation of mu}
\begin{split}
&\quad \E \{\mu_N^{u-}(t_0, \ep)\mathbbm{1}_{\{f(t_0)>u\}} \} =\int_u^\infty \E\{\mu_N^{u-}(t_0, \ep)|f(t_0)=x\}p_{f(t_0)}(x) dx\\
&= \int_u^\infty p_{f(t_0)}(x) dx\int_{U_{t_0}(\ep)} p_{\nabla f(t)}(0|f(t_0)=x) \\
&\quad \times \E\{|{\rm det} \nabla^2 f(t)|\mathbbm{1}_{\{{\rm index}(\nabla^2 f(t))=N\}}\mathbbm{1}_{\{f(t)\leq u\}} | \nabla f(t)=0, f(t_0)=x\}dt.
\end{split}
\end{equation}
By $({\bf C}1)$ and $({\bf C}2)$, for small $\ep>0$,
\begin{equation*}
\begin{split}
\sup_{t\in U_{t_0}(\ep)}p_{\nabla f(t)}(0|f(t_0)=x)\leq \sup_{t\in U_{t_0}(\ep)} \frac{1}{(2\pi)^{N/2} ({\rm det Cov}(\nabla f(t)|f(t_0)))^{1/2}}\leq C
\end{split}
\end{equation*}
for some positive constant $C$. On the other hand, by continuity, conditioning on $f(t_0)=x>u$, $\sup_{t\in U_{t_0}(\ep)}\mathbbm{1}_{\{f(t)\leq u\}}$ tends to 0 a.s. as $\ep \to 0$. Therefore, for each $x>u$, by the dominated convergence theorem (we may choose $\sup_{t\in U_{t_0}(\ep_0)} |{\rm det} \nabla^2 f(t)|$ as the dominating function for some $\ep_0>0$), as $\ep \to 0$,
\begin{equation*}
\begin{split}
\sup_{t\in U_{t_0}(\ep)}\E\{|{\rm det} \nabla^2 f(t)|\mathbbm{1}_{\{{\rm index}(\nabla^2 f(t))=N\}}\mathbbm{1}_{\{f(t)\leq u\}} | \nabla f(t)=0, f(t_0)=x\} \to 0.
\end{split}
\end{equation*}
Plugging these facts into (\ref{Eq:conditional expectation of mu}) and applying the dominated convergence theorem, we obtain that as $\ep\to 0$,
\begin{equation*}
\begin{split}
&\frac{1}{\ep^N}\E \{\mu_N^{u-}(t_0, \ep)\mathbbm{1}_{\{f(t_0)>u\}} \}\\
&\quad \le \frac{C}{\ep^N} \int_u^\infty \sup_{t\in U_{t_0}(\ep)}\E\{|{\rm det} \nabla^2 f(t)|\mathbbm{1}_{\{{\rm index}(\nabla^2 f(t))=N\}}\mathbbm{1}_{\{f(t)\leq u\}} | \nabla f(t)=0, f(t_0)=x\}\\
&\qquad \qquad \qquad \quad  \times p_{f(t_0)}(x) dx\int_{U_{t_0}(\ep)} dt\\
&\quad \to 0,
\end{split}
\end{equation*}
which implies $\E \{\mu_N^{u-}(t_0, \ep)\mathbbm{1}_{\{f(t_0)>u\}} \} = o(\ep^N)$. By (\ref{Eq:the first small prob}),
\begin{equation*}
\P\{f(t_0)>u, \mu_N(t_0, \ep)\geq 1, \mu_N^u(t_0, \ep) = 0\} = o(\ep^N).
\end{equation*}
Similar arguments yield
\begin{equation*}
\P\{f(t_0)\leq u, \mu_N^u(t_0, \ep) \geq 1 \} = o(\ep^N).
\end{equation*}
Hence (\ref{Eq:difference in numerator}) holds and therefore,
\begin{equation}\label{Eq:limiting form of Palm distr}
\begin{split}
F_{t_0}(u) &=\lim_{\ep\to 0} \frac{ \P\{f(t_0)>u, \mu_N(t_0, \ep)\geq 1\}}{\P\{ \mu_N(t_0, \ep)\geq 1\}} =\lim_{\ep\to 0} \frac{ \P\{\mu_N^u(t_0, \ep) \geq 1 \} + o(\ep^N)}{\P\{ \mu_N(t_0, \ep)\geq 1\}}\\
&=\lim_{\ep\to 0} \frac{ \E\{\mu_N^u(t_0, \ep)  \} + o(\ep^N)}{\E\{ \mu_N(t_0, \ep)\} + o(\ep^N)},
\end{split}
\end{equation}
where the last equality is due to (\ref{Eq:estimate prob of local max}) and (\ref{Eq:estimate prob of local max above u}).
By the Kac-Rice metatheorem and Lebesgue's continuity theorem,
\begin{equation*}
\begin{split}
&\lim_{\ep\to 0}\frac{1}{\ep^N}\E\{\mu_N^u(t_0, \ep)\} \\
&\quad =\lim_{\ep\to 0}
\frac{1}{\ep^N}\int_{U_{t_0}(\ep)} \E\{|{\rm det} \nabla^2 f(t)|\mathbbm{1}_{\{f(t)> u\}} \mathbbm{1}_{\{{\rm index}(\nabla^2 f(t))=N\}}|\nabla f(t)=0\}p_{\nabla f(t)}(0)dt\\
&\quad = \E\{|{\rm det} \nabla^2 f(t_0)|\mathbbm{1}_{\{f(t_0)> u\}} \mathbbm{1}_{\{{\rm index}(\nabla^2 f(t_0))=N\}}|\nabla f(t_0)=0\}p_{\nabla f(t_0)}(0),
\end{split}
\end{equation*}
and similarly,
\begin{equation*}
\begin{split}
\lim_{\ep\to 0}\frac{1}{\ep^N}\E\{\mu_N(t_0, \ep)\} =\E\{|{\rm det} \nabla^2 f(t_0)|\mathbbm{1}_{\{{\rm index}(\nabla^2 f(t_0))=N\}} | \nabla f(t_0)=0\}p_{\nabla f(t_0)}(0).
\end{split}
\end{equation*}
Plugging these into (\ref{Eq:limiting form of Palm distr}) yields (\ref{Eq:Palm distr Euclidean}).
\end{proof}

\begin{proof}{\bf of Theorem \ref{Thm:Palm distr high level}}\ By Theorem \ref{Thm:overshoot distr},
\begin{equation}\label{Eq:Palm distri u and v}
\begin{split}
\bar{F}_{t_0}(u,v)= \frac{\int_{u+v}^\infty \phi(x)\E\{|{\rm det} \nabla^2 f(t_0)|\mathbbm{1}_{\{{\rm index}(\nabla^2 f(t_0))=N\}}| f(t_0)=x, \nabla f(t_0)=0\} dx}{\int_v^\infty \phi(x)\E\{|{\rm det} \nabla^2 f(t_0)|\mathbbm{1}_{\{{\rm index}(\nabla^2 f(t_0))=N\}}| f(t_0)=x, \nabla f(t_0)=0\} dx}.
\end{split}
\end{equation}
We shall estimate the conditional expectations above. Note that $f$ has unit-variance, taking derivatives gives
\begin{equation*}
\E\{ f(t_0)\nabla^2 f(t_0) \} = -{\rm Cov}(\nabla f(t_0)) = -\La(t_0).
\end{equation*}
Since $\La(t_0)$ is positive definite, there exists a unique positive definite matrix $Q_{t_0}$ such that $Q_{t_0}\La(t_0)Q_{t_0}= I_N$ ($Q_{t_0}$ is also called the square root of $\La(t_0)$), where $I_N$ is the $N\times N$ unit matrix. Hence
\begin{equation*}
\E\{f(t_0)(Q_{t_0} \nabla^2 f(t_0) Q_{t_0})\} = -Q_{t_0}\La(t_0)Q_{t_0} = -I_N.
\end{equation*}
By the conditional formula for Gaussian random variables,
\begin{equation*}
\E\{Q_{t_0} \nabla^2 f(t_0) Q_{t_0} | f(t_0)=x, \nabla f(t_0)=0 \} = -xI_N.
\end{equation*}
Make change of variable
$$W(t_0) = Q_{t_0}\nabla^2 f(t_0)Q_{t_0} + xI_N,$$
where $W(t_0) = (W_{ij}(t_0))_{1\leq i, j\leq N}$. Then $(W(t_0)|f(t_0)=x, \nabla f(t_0)=0)$ is a Gaussian matrix
whose mean is 0 and covariance is the same as that of $(Q_{t_0} \nabla^2 f(t_0) Q_{t_0} | f(t_0)=x, \nabla f(t_0)=0)$.
Denote the density of Gaussian vector $((W_{ij}(t_0))_{1\leq i\leq j\leq N}|f(t_0)=x, \nabla f(t_0)=0)$ by $h_{t_0}(w)$,
$w=(w_{ij})_{1\leq i\leq j\leq N}\in \R^{N(N+1)/2}$, then
\begin{equation}\label{Eq:expectation of det}
\begin{split}
\E & \{\text{det} (Q_{t_0}\nabla^2 f(t_0)Q_{t_0}) \mathbbm{1}_{\{{\rm index}(\nabla^2 f(t_0))= N\}} | f(t_0)=x, \nabla f(t_0)=0 \}\\
& = \E \{\text{det} (Q_{t_0}\nabla^2 f(t_0)Q_{t_0}) \mathbbm{1}_{\{{\rm index}(Q_{t_0}\nabla^2 f(t_0)Q_{t_0})= N\}} | f(t_0)=x, \nabla f(t_0)=0 \} \\
& = \int_{w:\, {\rm index}((w_{ij})-xI_N) =N} \text{det} \Big((w_{ij})-xI_N \Big) h_{t_0}(w) \, dw,
\end{split}
\end{equation}
where $(w_{ij})$ is the abbreviation of matrix $(w_{ij})_{1\leq i, j\leq k}$.
Note that there exists a constant $c>0$ such that
$${\rm index}((w_{ij})-xI_N) =N, \quad \forall  \|(w_{ij})\|:= \bigg(\sum_{i,j=1}^N w_{ij}^2\bigg)^{1/2} <\frac{x}{c}.$$
Thus we can write (\ref{Eq:expectation of det}) as
\begin{equation}\label{Eq:expectation of det 2}
\begin{split}
&\int_{\R^{N(N+1)/2}} \text{det} \Big((w_{ij})-xI_N \Big) h_{t_0}(w)  dw - \int_{w:\, {\rm index}((w_{ij})-xI_N) <N} \text{det} \Big((w_{ij})-xI_N \Big) h_{t_0}(w) \, dw\\
& \quad = \E\{\text{det} (Q_{t_0}\nabla^2 f(t_0)Q_{t_0}) | f(t_0)=x, \nabla f(t_0)=0 \} + Z(t,x),
\end{split}
\end{equation}
where $Z(t,x)$ is the second integral in the first line of (\ref{Eq:expectation of det 2}) and it satisfies
\begin{equation}\label{Eq:estimate Z(t,x)}
|Z(t,x)| \leq \int_{\|(w_{ij})\|\geq\frac{x}{c}} \bigg|\text{det} \Big((w_{ij})-xI_N \Big)\bigg| h_{t_0}(w)  dw.
\end{equation}
By the non-degenerate condition ({\bf C}2), there exists a constant $\alpha'>0$ such that as $\|(w_{ij})\| \to \infty$,  $h_{t_0}(w) = o(e^{-\alpha' \|(w_{ij})\|^2})$. On the other hand, the determinant inside the integral in (\ref{Eq:estimate Z(t,x)}) is a polynomial in $w_{ij}$ and $x$, and it does not affect the exponentially decay, hence as $x\to \infty$, $|Z(t,x)| = o(e^{-\alpha x^2})$ for some constant $\alpha>0$. Combine this with (\ref{Eq:expectation of det}) and (\ref{Eq:expectation of det 2}), and note that
$$
{\rm det} \nabla^2 f(t_0) = {\rm det} (Q_{t_0}^{-1}Q_{t_0}\nabla^2 f(t_0)Q_{t_0}Q_{t_0}^{-1})={\rm det}(\La(t_0)){\rm det} (Q_{t_0}\nabla^2 f(t_0)Q_{t_0}),
$$
we obtain that, as $x\to \infty$,
\begin{equation*}
\begin{split}
\E&\{|{\rm det} \nabla^2 f(t_0)|\mathbbm{1}_{\{{\rm index}(\nabla^2 f(t_0))=N\}}| f(t_0)=x, \nabla f(t_0)=0\}\\
& = (-1)^N{\rm det}(\La(t_0))\E\{{\rm det} (Q_{t_0}\nabla^2 f(t_0)Q_{t_0})\mathbbm{1}_{\{{\rm index}(\nabla^2 f(t_0))=N\}}| f(t_0)=x, \nabla f(t_0)=0\}\\
& = (-1)^N{\rm det}(\La(t_0))\E\{{\rm det} (Q_{t_0}\nabla^2 f(t_0)Q_{t_0})| f(t_0)=x, \nabla f(t_0)=0\} + o(e^{-\alpha x^2})\\
& = (-1)^N\E\{{\rm det} \nabla^2 f(t_0)| f(t_0)=x, \nabla f(t_0)=0\} + o(e^{-\alpha x^2})
\end{split}
\end{equation*}
Plugging this into (\ref{Eq:Palm distri u and v}) yields (\ref{Eq:Palm distr high level}).
\end{proof}

\begin{lemma}\label{Lem:conditional expectation for N} Under the assumptions in Theorem \ref{Thm:Palm distr high level}, as $x\to \infty$,
\begin{equation}\label{Eq:det hessian conditional}
\begin{split}
\E\{{\rm det} \nabla^2 f(t) | f(t)=x, \nabla f(t)=0 \}=(-1)^N {\rm det}(\La(t)) x^N (1+O(x^{-2})).
\end{split}
\end{equation}
\end{lemma}
\begin{proof}\ Let $Q_t$ be the $N\times N$ positive definite matrix such that $Q_t\La(t)Q_t = I_N$. Then we can write $\nabla^2 f(t)= Q^{-1}_tQ_t\nabla^2 f(t)Q_tQ^{-1}_t$ and therefore,
\begin{equation}\label{Eq:det hessian conditional 1}
\begin{split}
\E&\{{\rm det} \nabla^2 f(t) | f(t)=x, \nabla f(t)=0 \}\\
&={\rm det}(\La(t)) \E\{{\rm det} (Q_t\nabla^2 f(t)Q_t) | f(t)=x, \nabla f(t)=0 \}.
\end{split}
\end{equation}
Since $f(t)$ and $\nabla f(t)$ are independent,
$$
\E\{Q_t \nabla^2 f(t) Q_t | f(t)=x, \nabla f(t)=0 \} = -xI_N.
$$
It follows that
\begin{equation}\label{Eq:det hessian conditional 2}
\begin{split}
\E\{Q_t \nabla^2 f(t) Q_t | f(t)=x, \nabla f(t)=0 \}= \E\{{\rm det} (\wt{\Delta}(t) - xI_N) \},
\end{split}
\end{equation}
where $\wt{\Delta}(t)=(\wt{\Delta}_{ij}(t))_{1\le i,j\le N}$ is an $N\times N$ Gaussian random matrix such that $\E\{\wt{\Delta}(t)\} =0$ and its covariance matrix is independent of $x$. By the Laplace expansion of the determinant,
\begin{equation*}
\begin{split}
{\rm det} (\wt{\Delta}(t) - xI_N) = (-1)^N[x^N - S_1(\wt{\Delta}(t))x^{N-1} +  S_2(\wt{\Delta}(t))x^{N-2} + \cdots + (-1)^N S_N(\wt{\Delta}(t))],
\end{split}
\end{equation*}
where $S_i(\wt{\Delta}(t))$ is the sum of the $\binom {N}{i}$ principle minors of order $i$ in $\wt{\Delta}(t)$. Taking the expectation above and noting that $\E\{S_1(\wt{\Delta}(t))\}=0$ since $\E\{\wt{\Delta}(t)\} =0$, we obtain that as $x\to \infty$,
\begin{equation*}
\begin{split}
\E\{{\rm det} (\wt{\Delta}(t) - xI_N) \} = (-1)^Nx^N (1+O(x^{-2})).
\end{split}
\end{equation*}
Combining this with (\ref{Eq:det hessian conditional 1}) and (\ref{Eq:det hessian conditional 2}) yields (\ref{Eq:det hessian conditional}).
\end{proof}

\begin{lemma}\label{Lem:cov of isotropic Euclidean}{\rm [Aza\"is and Wschebor (2008), Lemma 2].} Let $\{f(t): t\in T\}$ be a centered, unit-variance, isotropic Gaussian random field satisfying $({\bf C}1)$ and $({\bf C}2)$. Then for each $t\in T$ and $i$, $j$, $k$, $l\in\{1,\ldots, N\}$,
\begin{equation}\label{Eq:cov of derivatives}
\begin{split}
&\E\{f_i(t)f(t)\}=\E\{f_i(t)f_{jk}(t)\}=0, \quad \E\{f_i(t)f_j(t)\}=-\E\{f_{ij}(t)f(t)\}=-2\rho'\delta_{ij},\\
&\E\{f_{ij}(t)f_{kl}(t)\}=4\rho''(\delta_{ij}\delta_{kl} + \delta_{ik}\delta_{jl} + \delta_{il}\delta_{jk}),
\end{split}
\end{equation}
where $\delta_{ij}$ is the Kronecker delta and $\rho'$ and $\rho''$ are defined in \eqref{Eq:kappa}.
\end{lemma}

\begin{lemma}\label{Lem:GOE for det Hessian} Under the assumptions in Lemma \ref{Lem:cov of isotropic Euclidean}, the distribution of $\nabla^2f(t)$ is the same as that of $\sqrt{8\rho''}M_N + 2\sqrt{\rho''}\xi I_N$, where $M_N$ is a GOE random matrix, $\xi$ is a standard Gaussian variable independent of $M_N$ and $I_N$ is the $N\times N$ identity matrix. Assume further that $({\bf C}3)$ holds, then the conditional distribution of $(\nabla^2f(t)|f(t)=x)$ is the same as the distribution of $\sqrt{8\rho''}M_N + [2\rho'x +2\sqrt{\rho''-\rho'^2}\xi ]I_N$.
\end{lemma}
\begin{proof} \ The first result is a direct consequence of Lemma \ref{Lem:cov of isotropic Euclidean}. For the second one, applying (\ref{Eq:cov of derivatives}) and the well-known conditional formula for Gaussian variables, we see that $(\nabla^2f(t)|f(t)=x)$ can be written as $\Delta + 2\rho'xI_N$, where $\Delta=(\Delta_{ij})_{1\leq i,j\leq N}$ is a symmetric $N\times N$ matrix with centered Gaussian entries such that
\begin{equation*}
\E\{\Delta_{ij}\Delta_{kl}\}=4\rho''(\delta_{ik}\delta_{jl} + \delta_{il}\delta_{jk}) + 4(\rho''-\rho'^2)\delta_{ij}\delta_{kl}.
\end{equation*}
Therefore, $\Delta$ has the same distribution as the random matrix $\sqrt{8\rho''}M_N + 2\sqrt{\rho''-\rho'^2}\xi I_N$, completing the proof.
\end{proof}

Lemma \ref{Lem:GOE computation} below is a revised version of Lemma 3.2.3 in Auffinger (2011). The proof is omitted here since it is similar to that of the reference above.
\begin{lemma}\label{Lem:GOE computation} Let $M_N$ be an $N\times N$ GOE matrix and $X$ be an independent Gaussian random variable with mean $m$ and variance $\sigma^2$. Then,
\begin{equation}\label{Eq:computing the det with index}
\begin{split}
&\E\big\{|{\rm det}(M_N-XI_N)|\mathbbm{1}_{\{{\rm index}(M_N-XI_N)=N\}}\big\}\\
&\quad =\frac{\Gamma\big(\frac{N+1}{2}\big)}{\sqrt{\pi}\sigma}\E_{GOE}^{N+1}\bigg\{ \exp\bigg[\frac{\la_{N+1}^2}{2} - \frac{(\la_{N+1}-m )^2}{2\sigma^2} \bigg] \bigg\},
\end{split}
\end{equation}
where $\E_{GOE}^{N+1}$ is the expectation under the probability distribution $Q_{N+1}(d\la)$ as in (\ref{Eq:GOE density}) with $N$ replaced by $N+1$.
\end{lemma}

\begin{lemma}\label{Lem:expectation of local max} Let $\{f(t): t\in T\}$ be a centered, unit-variance, isotropic Gaussian random field satisfying $({\bf C}1)$ and $({\bf C}2)$. Then for each $t\in T$,
\begin{equation*}
\begin{split}
\E&\{|{\rm det} \nabla^2 f(t)|\mathbbm{1}_{\{{\rm index}(\nabla^2 f(t))=N\}}| \nabla f(t)=0\}\\
&=\Big(\frac{2}{\pi}\Big)^{1/2}\Gamma\Big(\frac{N+1}{2}\Big)(8\rho'')^{N/2}\E_{GOE}^{N+1}\bigg\{ \exp\bigg[-\frac{\la_{N+1}^2}{2} \bigg] \bigg\}.
\end{split}
\end{equation*}
\end{lemma}
\begin{proof}\ Since $\nabla^2 f(t)$ and $\nabla f(t)$ are independent for each fixed $t$, by Lemma \ref{Lem:GOE for det Hessian},
\begin{equation*}
\begin{split}
\E&\{|{\rm det} \nabla^2 f(t)|\mathbbm{1}_{\{{\rm index}(\nabla^2 f(t))=N\}}| \nabla f(t_0)=0\}\\
&= \E\{|{\rm det} \nabla^2 f(t)|\mathbbm{1}_{\{{\rm index}(\nabla^2 f(t))=N\}}\}\\
&= \E\{|{\rm det} (\sqrt{8\rho''}M_N + 2\sqrt{\rho''}\xi I_N)|\mathbbm{1}_{\{{\rm index}(\sqrt{8\rho''}M_N + 2\sqrt{\rho''}\xi I_N)=N\}}\}\\
&= (8\rho'')^{N/2} \E\{|{\rm det}(M_N-XI_N)|\mathbbm{1}_{\{{\rm index}(M_N-XI_N)=N\}}\},
\end{split}
\end{equation*}
where $X$ is an independent centered Gaussian variable with variance $1/2$. Applying Lemma \ref{Lem:GOE computation} with $m=0$ and $\sigma=1/\sqrt{2}$, we obtain the desired result.
\end{proof}

\begin{lemma}\label{Lem:expectation of local max above u} Let $\{f(t): t\in T\}$ be a centered, unit-variance, isotropic Gaussian random field satisfying $({\bf C}1)$, $({\bf C}2)$ and $({\bf C}3)$. Then for each $t\in T$ and $x\in \R$,
\begin{equation*}
\begin{split}
\E&\{|{\rm det} \nabla^2 f(t)| \mathbbm{1}_{\{{\rm index}(\nabla^2 f(t))=N\}}|f(t)=x, \nabla f(t)=0\}\\
&=\left\{
  \begin{array}{l l}
     \big(\frac{2}{\pi}\big)^{1/2} \Gamma\big(\frac{N+1}{2}\big)(8\rho'')^{N/2} \big(\frac{\rho''}{\rho''-\rho'^2}\big)^{1/2}\\
     \quad \times \E_{GOE}^{N+1}\bigg\{ \exp\bigg[\frac{\la_{N+1}^2}{2} - \frac{\rho''\big(\la_{N+1}+\frac{\rho'x}{\sqrt{2\rho''}} \big)^2}{\rho''-\rho'^2} \bigg]\bigg\} & \quad \text{if $\rho''-\rho'^2> 0$},\\
     (8\rho'')^{N/2}\E_{GOE}^{N}\big\{ \big(\prod_{i=1}^N|\la_i-\frac{x}{\sqrt{2}}|\big) \mathbbm{1}_{\{\la_N<\frac{x}{\sqrt{2}}\}} \big\} & \quad \text{if $\rho''-\rho'^2= 0$}.
   \end{array} \right.
\end{split}
\end{equation*}
\end{lemma}
\begin{proof}\ Since $\nabla f(t)$ is independent of both $f(t)$ and $\nabla^2 f(t)$ for each fixed $t$, by Lemma \ref{Lem:GOE for det Hessian},
\begin{equation}\label{Eq:Conditional expectation by GOE}
\begin{split}
\E&\{|{\rm det} \nabla^2 f(t)| \mathbbm{1}_{\{{\rm index}(\nabla^2 f(t))=N\}}|f(t)=x, \nabla f(t)=0\}\\
&=\E\{|{\rm det} \nabla^2 f(t)|\mathbbm{1}_{\{{\rm index}(\nabla^2 f(t))=N\}}| f(t)=x\}\\
&= \E\{|{\rm det} (\sqrt{8\rho''}M_N + [2\rho'x +2\sqrt{\rho''-\rho'^2}\xi ]I_N)|\\
&\qquad \times \mathbbm{1}_{\{{\rm index}(\sqrt{8\rho''}M_N + [2\rho'x +2\sqrt{\rho''-\rho'^2}\xi ]I_N)=N\}}\}.
\end{split}
\end{equation}
When $\rho''-\rho'^2> 0$, then (\ref{Eq:Conditional expectation by GOE}) can be written as
\begin{equation*}
\begin{split}
(8\rho'')^{N/2} \E\{|{\rm det}(M_N-XI_N)|\mathbbm{1}_{\{{\rm index}(M_N-XI_N)=N\}}\},
\end{split}
\end{equation*}
where $X$ is an independent Gaussian variable with mean $m=-\frac{\rho'x}{\sqrt{2\rho''}}$ and variance $\sigma^2= \frac{\rho''-\rho'^2}{2\rho''}$. Applying Lemma \ref{Lem:GOE computation} yields the formula for the case of $\rho''-\rho'^2> 0$.

When $\rho''-\rho'^2= 0$, i.e. $\rho'=-\sqrt{\rho''}$, then (\ref{Eq:Conditional expectation by GOE}) becomes
\begin{equation*}
\begin{split}
&(8\rho'')^{N/2} \E\{|{\rm det}(M_N-\frac{x}{\sqrt{2}}I_N)|\mathbbm{1}_{\{{\rm index}(M_N-\frac{x}{\sqrt{2}}I_N)=N\}}\}\\
&\quad =(8\rho'')^{N/2}\E_{GOE}^{N}\bigg\{ \bigg(\prod_{i=1}^N|\la_i-\frac{x}{\sqrt{2}}|\bigg) \mathbbm{1}_{\{\la_N<\frac{x}{\sqrt{2}}\}} \bigg\}.
\end{split}
\end{equation*}
We finish the proof.
\end{proof}

The following result can be derived from elementary calculations by applying the GOE density \eqref{Eq:GOE density}, the details are omitted here.
\begin{proposition}\label{Prop:GOE expectation for N=2} Let $N=2$. Then for positive constants $a$ and $b$,
\begin{equation*}
\begin{split}
&\E_{GOE}^{N+1}\bigg\{ \exp\Big[\frac{1}{2}\la_{N+1}^2 -a(\la_{N+1}-b)^2 \Big] \bigg\} \\
&\quad=\frac{1}{\sqrt{2}\pi}\bigg\{(\frac{1}{a} + 2b^2-1)\frac{\pi}{\sqrt{a}}\Phi\Big(\frac{b\sqrt{2a}}{\sqrt{a+1}} \Big) + \frac{b\sqrt{a+1}\sqrt{\pi}}{a}e^{-\frac{ab^2}{a+1}} \\
&\qquad + \frac{2\pi}{\sqrt{2a+1}}e^{-\frac{ab^2}{2a+1}}\Phi\Big(\frac{\sqrt{2}ab}{\sqrt{(2a+1)(a+1)}} \Big) \bigg\}.
\end{split}
\end{equation*}
\end{proposition}

\subsection{Proofs for Section \ref{section:general manifolds}}
Define $\mu(t_0, \ep)$, $\mu_N(t_0, \ep)$, $\mu_N^u(t_0, \ep)$ and $\mu_N^{u-}(t_0, \ep)$ as in (\ref{Def:various mu's}) with $U_{t_0}(\ep)$ replaced by $B_{t_0}(\ep)$ respectively. The following lemma, which will be used for proving Theorem \ref{Thm:Palm distr manifolds}, is an analogue of Lemma \ref{Lem:Piterbarg}.
\begin{lemma}\label{Lem:Piterbarg lemma on manifold} Let $(M,g)$ be an oriented $N$-dimensional $C^3$ Riemannian manifold with a $C^1$ Riemannian metric $g$.  Let $f$ be a Gaussian random field on $M$ such that $({\bf C}1')$ and $({\bf C}2')$ are fulfilled. Then for any $t_0 \in \overset{\circ}{M}$, as $\ep \to 0$,
\begin{equation*}
\begin{split}
\E \{\mu(t_0, \ep)(\mu(t_0, \ep) -1)\} = o(\ep^N).
\end{split}
\end{equation*}
\end{lemma}
\begin{proof}\ Let $(U_\alpha, \varphi_\alpha)_{\alpha\in I}$ be an atlas on $M$ and let $\ep$ be small enough such that $\varphi_\alpha(B_{t_0}(\ep)) \subset \varphi_\alpha(U_\alpha)$ for some $\alpha\in I$. Set
$$
f^\alpha= f\circ \varphi_\alpha^{-1}: \varphi_\alpha(U_\alpha) \subset \R^N \rightarrow \R.
$$
Then it follows immediately from the diffeomorphism of $\varphi_\alpha$ and the definition of $\mu$ that
\begin{equation*}
\mu(t_0, \ep)=\mu(f, U_\alpha; t_0, \ep) \equiv \mu(f^\alpha, \varphi_\alpha(U_\alpha); \varphi_\alpha(t_0), \ep).
\end{equation*}
Note that $({\bf C}1')$ and $({\bf C}2')$ imply that $f^\alpha$ satisfies $({\bf C}1)$ and $({\bf C}2)$. Applying Lemma \ref{Lem:Piterbarg} gives
\begin{equation*}
\E \{\mu(f^\alpha, \varphi_\alpha(U_\alpha); \varphi_\alpha(t_0), \ep)[\mu(f^\alpha, \varphi_\alpha(U_\alpha); \varphi_\alpha(t_0), \ep)-1]\} = o({\rm Vol}(\varphi_\alpha(B_{t_0}(\ep))))=o(\ep^N).
\end{equation*}
This verifies the desired result.
\end{proof}

\begin{proof}{\bf of Theorem \ref{Thm:Palm distr manifolds}}\ Following the proof in Theorem \ref{Thm:Palm distr}, together with Lemma \ref{Lem:Piterbarg lemma on manifold} and the argument by charts in its proof, we obtain
\begin{equation}\label{Eq:limiting form of Palm distr 2}
\begin{split}
F_{t_0}(u) &=\lim_{\ep\to 0} \frac{ \P\{f(t_0)>u, \mu_N(t_0, \ep)\geq 1\}}{\P\{ \mu_N(t_0, \ep)\geq 1\}} =\lim_{\ep\to 0} \frac{ \P\{\mu_N^u(t_0, \ep) \geq 1 \} + o(\ep^N)}{\P\{ \mu_N(t_0, \ep)\geq 1\}}\\
&=\lim_{\ep\to 0} \frac{ \E\{\mu_N^u(t_0, \ep)  \} + o(\ep^N)}{\E\{ \mu_N(t_0, \ep)\} + o(\ep^N)}.
\end{split}
\end{equation}
By the Kac-Rice metatheorem for random fields on manifolds [cf. Theorem 12.1.1 in Adler and Taylor (2007)] and Lebesgue's continuity theorem,
\begin{equation*}
\begin{split}
&\lim_{\ep\to 0}\frac{\E\{\mu_N^u(t_0, \ep)\}}{{\rm Vol}(B_{t_0}(\ep))} \\
&\quad=\lim_{\ep\to 0}
\frac{1}{{\rm Vol}(B_{t_0}(\ep))}\int_{B_{t_0}(\ep)} \E\{|{\rm det} \nabla^2 f(t)|\mathbbm{1}_{\{f(t)> u\}} \mathbbm{1}_{\{{\rm index}(\nabla^2 f(t))=N\}}|\nabla f(t)=0\}\\
&\qquad \qquad\qquad \qquad\qquad\qquad \times p_{\nabla f(t)}(0){\rm Vol}_g\\
&\quad = \E\{|{\rm det} \nabla^2 f(t_0)|\mathbbm{1}_{\{f(t_0)> u\}} \mathbbm{1}_{\{{\rm index}(\nabla^2 f(t_0))=N\}}|\nabla f(t_0)=0\}p_{\nabla f(t_0)}(0),
\end{split}
\end{equation*}
where ${\rm Vol}_g$ is the volume element on $M$ induced by the Riemannian metric $g$. Similarly,
\begin{equation*}
\begin{split}
\lim_{\ep\to 0}\frac{\E\{\mu_N(t_0, \ep)\} }{{\rm Vol}(B_{t_0}(\ep))}= \E\{|{\rm det} \nabla^2 f(t_0)|\mathbbm{1}_{\{{\rm index}(\nabla^2 f(t_0))=N\}}|\nabla f(t_0)=0\}p_{\nabla f(t_0)}(0).
\end{split}
\end{equation*}
Plugging these facts into (\ref{Eq:limiting form of Palm distr 2}) yields the first line of \eqref{Eq:Palm distr manifolds 1}. The second line of \eqref{Eq:Palm distr manifolds 1} follows similarly.

Applying Theorem \ref{Thm:Palm distr high level} and Corollary \ref{Cor:Palm distr high level o(1)}, together with the argument by charts, we obtain \eqref{Eq:Palm distr manifolds 3}.
\end{proof}

Lemma \ref{Lem:joint distribution sphere} below is on the properties of the covariance of $(f(t), \nabla f(t), \nabla^2 f(t))$, where the gradient $\nabla f(t)$ and Hessian $\nabla^2 f(t)$ are defined as in (\ref{Eq:gradient and hessian on manifolds}) under some orthonormal frame $\{E_i\}_{1\le i\le N}$ on $\mathbb{S}^N$. Since it can be proved similarly to Lemma 3.2.2 or Lemma 4.4.2 in Auffinger (2011), the detailed proof is omitted here.
\begin{lemma}\label{Lem:joint distribution sphere}
Let $f$ be a centered, unit-variance, isotropic Gaussian field on $\mathbb{S}^N$, $N\ge 2$, satisfying $({\bf C}1'')$ and $({\bf C}2')$. Then
\begin{equation}\label{Eq:cov of derivatives sphere}
\begin{split}
&\E\{f_i(t)f(t)\}=\E\{f_i(t)f_{jk}(t)\}=0, \quad \E\{f_i(t)f_j(t)\}=-\E\{f_{ij}(t)f(t)\}=C'\delta_{ij},\\
&\E\{f_{ij}(t)f_{kl}(t)\}=C''(\delta_{ik}\delta_{jl} + \delta_{il}\delta_{jk}) + (C''+C')\delta_{ij}\delta_{kl},
\end{split}
\end{equation}
where $C'$ and $C''$ are defined in \eqref{Def:C' and C''}.
\end{lemma}

\begin{lemma}\label{Lem:GOE for det Hessian sphere}
Under the assumptions in Lemma \ref{Lem:joint distribution sphere}, the distribution of $\nabla^2f(t)$ is the same as that of $\sqrt{2C''}M_N + \sqrt{C''+C'}\xi I_N$, where $M_N$ is a GOE matrix and $\xi$ is a standard Gaussian variable independent of $M_N$. Assume further that $({\bf C}3')$ holds, then the conditional distribution of $(\nabla^2f(t)|f(t)=x)$ is the same as the distribution of $\sqrt{2C''}M_N + [\sqrt{C''+C'-C'^2}\xi - C'x]I_N$.
\end{lemma}
\begin{proof}\
The first result is an immediate consequence of Lemma \ref{Lem:joint distribution sphere}. For the second one, applying (\ref{Eq:cov of derivatives sphere}) and the well-known conditional formula for Gaussian random variables, we see that $(\nabla^2f(t)|f(t)=x)$ can be written as $\Delta - C'xI_N$, where $\Delta=(\Delta_{ij})_{1\leq i,j\leq N}$ is a symmetric $N\times N$ matrix with centered Gaussian entries such that
\begin{equation*}
\E\{\Delta_{ij}\Delta_{kl}\}=C''(\delta_{ik}\delta_{jl} + \delta_{il}\delta_{jk}) + (C''+C'-C'^2)\delta_{ij}\delta_{kl}.
\end{equation*}
Therefore, $\Delta$ has the same distribution as the random matrix $\sqrt{2C''}M_N + \sqrt{C''+C'-C'^2}\xi I_N$, completing the proof.
\end{proof}

By similar arguments for proving Lemmas \ref{Lem:expectation of local max} and \ref{Lem:expectation of local max above u}, and applying Lemma \ref{Lem:GOE for det Hessian sphere} instead of Lemma \ref{Lem:GOE for det Hessian},  we obtain the following two lemmas.
\begin{lemma}\label{Lem:expectation of local max sphere} Let $\{f(t): t\in \mathbb{S}^N\}$ be a centered, unit-variance, isotropic Gaussian field satisfying $({\bf C}1'')$ and $({\bf C}2')$. Then for each $t\in \mathbb{S}^N$,
\begin{equation*}
\begin{split}
\E&\{|{\rm det} \nabla^2 f(t)|\mathbbm{1}_{\{{\rm index}(\nabla^2 f(t))=N\}}| \nabla f(t)=0\}\\
&=\Big(\frac{2C''}{\pi(C''+C')}\Big)^{1/2} \Gamma\Big(\frac{N+1}{2}\Big) (2C'')^{N/2}\E_{GOE}^{N+1}\bigg\{ \exp\bigg[\frac{1}{2}\la_{N+1}^2 -\frac{C''}{C''+C'}\la_{N+1}^2 \bigg] \bigg\}.
\end{split}
\end{equation*}
\end{lemma}

\begin{lemma}\label{Lem:expectation of local max above u sphere} Let $\{f(t): t\in \mathbb{S}^N\}$ be a centered, unit-variance, isotropic Gaussian field satisfying $({\bf C}1'')$, $({\bf C}2')$ and $({\bf C}3')$. Then for each $t\in \mathbb{S}^N$ and $x\in \R$,
\begin{equation*}
\begin{split}
\E&\{|{\rm det} \nabla^2 f(t)| \mathbbm{1}_{\{{\rm index}(\nabla^2 f(t))=N\}}|f(t)=x, \nabla f(t)=0\}\\
&=\left\{
  \begin{array}{l l}
     \big(\frac{2C''}{\pi(C''+C'-C'^2)}\big)^{1/2} \Gamma\big(\frac{N+1}{2}\big) (2C'')^{N/2} \\
     \quad \times \E_{GOE}^{N+1}\Big\{ \exp\Big[\frac{\la_{N+1}^2}{2} - \frac{C''\big(\la_{N+1}-\frac{C'x}{\sqrt{2C''}} \big)^2}{C''+C'-C'^2} \Big]\Big\} & \quad \text{if $C''+C'-C'^2> 0$},\\
     (2C'')^{N/2}\E_{GOE}^{N}\big\{ \big(\prod_{i=1}^N|\la_i-\frac{C'x}{\sqrt{2C''}}|\big) \mathbbm{1}_{\{\la_N<\frac{C'x}{\sqrt{2C''}}\}} \big\} & \quad \text{if $C''+C'-C'^2= 0$}.
   \end{array} \right.
\end{split}
\end{equation*}
\end{lemma}

\par\bigskip\noindent
{\bf Acknowledgments.} The authors thank Robert Adler of the Technion - Israel Institute of Technology for useful discussions and the anonymous referees for their insightful comments which have led to several improvements of this manuscript.

\bibliographystyle{plain}

\begin{small}

\end{small}

\end{document}